%% file: main.tex
\documentclass{TPmod}
\input{preamble.tex}

\input{mathpreamble.tex}

\begin{document}

\title{Integrality of mirror maps and arithmetic homological mirror symmetry for Greene--Plesser mirrors}

\author{Sheel Ganatra, Andrew Hanlon, Jeff Hicks, Daniel Pomerleano, and Nick Sheridan}

\begin{abstract}
We prove the `integrality of Taylor coefficients of mirror maps' conjecture for Greene--Plesser mirror pairs as a natural byproduct of an arithmetic refinement of homological mirror symmetry. 
We also prove homological mirror symmetry for Greene--Plesser mirror pairs in all characteristics such that the B-side family has good reduction, generalizing work of the fifth author and Smith over the complex numbers. A key technical ingredient is a new versality argument which allows us to work throughout over a Novikov-type ring with integer coefficients.
\end{abstract}

\maketitle

\section{Introduction}
\subsection{Integrality of mirror maps}
In the early days of mirror symmetry, it was conjectured that the coefficients of so-called `mirror maps' should be integers; see \cite[Conjecture 6.3.4]{Batyrev1995}, \cite{Lian1996}.
For example, let us consider the case of a smooth degree-$n$ hypersurface in $\mathbb{CP}^{n-1}$, which has one K\"ahler parameter $q$. 
The mirror is a crepant resolution of a quotient of the hypersurface 
$$\left\{\prod_{i=1}^n z_i = \nov \sum_{i=1}^n z_i^n\right\} \subset \mathbb{CP}^{n-1}$$
by a finite group and has one complex parameter $T = \nov^n$. 
The mirror map takes the form $q(T) = T \cdot \phi(T)^n$. Here
\begin{align}
\label{eq:hyp_mm}    \phi(T) &= \exp\left( \frac{\sum_{i \ge 1} F_i H_i T^i}{\sum_{ i \ge 0} F_i T^i}\right),\quad \text{where}\\
\nonumber    F_i &= \frac{(ni)!}{(i!)^n} \quad \text{and}\\
\nonumber    H_i &= \sum_{k=i+1}^{ni} \frac{1}{k}. 
\end{align} 

It was first proved by Lian--Yau that the power series $q(T)$ has integer coefficients when $n$ is prime \cite{Lian1998}; Zudilin proved the same statement in the case that $n$ is a prime power \cite{zudilin2002}; Lian--Yau extended their result to show that $\phi(T)$ has integer coefficients when $n$ is prime \cite{Lian-Yau-nthroot}; and Krattenthaler--Rivoal proved this extended result for general $n$ \cite{Krattenthaler2010}. All of the above results were obtained using methods of $p$-adic analysis. In \cite{Kontsevich2006}, Kontsevich--Schwarz--Vologodsky introduced an algebro-geometric approach to studying these integrality questions and applied them to the case of the quintic 3-fold.  

More generally, one can consider the case of a Batyrev mirror pair of Calabi--Yau hypersurfaces  \cite{Batyrev1993}.
The mirror map is computed and conjectured to have integral coefficients in \cite[Conjecture 6.3.4]{Batyrev1995} (see also \cite[Section 6.3.4]{coxkatz}). 
We now recall the conjecture.

Let $M$ be a lattice, $\Delta \subset M_\R$ a reflexive lattice polytope, and $P \subset \partial \Delta \cap M$ be the set of boundary lattice points which are not contained in the interior of a codimension-$1$ face. 
We consider the map $\Z^{P} \to M$ which sends the $\vec{p}$th generator to $\vec{p}$, let $K$ be its kernel, and
\begin{align*}
    K_{\vec{p}} &:= \{u \in K: u_{\vec{q}} \ge 0 \text{ for $\vec{q} \neq \vec{p}$}\}\\
    K_{\ge 0} & := \{u \in K: u_{\vec{q}} \ge 0 \text{ for all $\vec{q}$}\}.
\end{align*}
Let $K_+ \subset K$ be the submonoid generated by the $K_{\vec{p}}$.
By Lemma \ref{lem:K+inNEA}, the cone generated by $K_+ \subset \R^P$ is strongly convex, so we may define $\Z[[K_+]]$ to be the completed group ring. 
For any $u \in K_+$, we write 
$$\nov^u = \prod_{\vec{p} \in P} \nov_{\vec{p}}^{u_{\vec{p}}}.$$

We introduce the following notation for harmonic sums:
\begin{equation} H(k) := \sum_{i=1}^k \frac{1}{i}\end{equation}
whenever  $k \in \Z_{\ge 1}$, and we define $H(0) = 0$. 
We also define the `combinations' function:
\begin{align*}
\mathsf{comb}: (\Z_{\ge 0})^{P} & \to \Z_{\ge 1},\\
\mathsf{comb}(u) & := \frac{\left(\sum_{\vec{p} \in P} u_{\vec{p}}\right)!}{\prod_{\vec{p} \in P} u_{\vec{p}}!},
\end{align*}
and extend it to 
\begin{align*}
    \mathsf{comb}_{\vec{p}}: K_{\vec{p}} \setminus K_{\ge 0} & \to \Z,\\
    \mathsf{comb}_{\vec{p}}(u) & := (-1)^{u_{\vec{p}}+1}\frac{\left(\sum_{\vec{q} \in P} u_{\vec{q}}\right)! (-u_{\vec{p}}-1)!}{\prod_{\vec{q} \in P \setminus \{\vec{p}\}} u_{\vec{q}}!}.
\end{align*}
(Note that $\sum_{\vec{q} \in P} u_{\vec{q}} \ge 0$ for $u \in K_{\vec{p}}$ by \cite[Lemma 9.2]{BeukersVlasenkoIII}, so its factorial is defined.)

Define
\begin{align*} \tau(\nov) &:= \sum_{u \in K_{\ge 0}} \mathsf{comb}(u) \cdot \nov^u  \in \Z[[K_{\ge 0}]],\\
\tau_{\vec{p}}(\nov) &:= \sum_{u \in K_{\ge 0}} \mathsf{comb}(u) \cdot \left(H\left(\sum_{\vec{q} \in P}u_{\vec{q}}\right) - H(u_{\vec{p}})\right) \cdot \nov^u \in \Q[[K_{\ge 0}]],\\
\gamma_{\vec{p}}(\nov) &:= \sum_{u \in K_{\vec{p}} \setminus K_{\ge 0}} \mathsf{comb}_{\vec{p}}(u) \cdot \nov^u \in \Q[[K_{\vec{p}}]]
\end{align*}
for $\vec{p} \in P$.  
Finally, define
$$\phi_{\vec{p}}(\nov) := \exp\left(\frac{\tau_{\vec{p}}(\nov)+\gamma_{\vec{p}}(\nov)}{\tau(\nov)}\right) \in \Q[[K_+]].$$

\begin{conjmain}[Conjecture 6.3.4 of \cite{Batyrev1995}]\label{conj:intmap}
For any $u \in K$, we have
$$\prod_{\vec{p} \in P} \phi_{\vec{p}}(\nov)^{u_{\vec{p}}} \in \Z[[K_+]].$$
\end{conjmain}

\begin{rmk}\label{rmk:mirrmap}
    Let us make the connection with mirror maps explicit. 
    By Lemma \ref{lem:K+inNEA}, the cone generated by $(\Z_{\ge 0})^P + K_+$ is strongly convex, so we may define the ring $\C[[(\Z_{\ge 0})^P+K_+]]$. 
    Let us consider the map 
\begin{align}
\label{eq:mm_formula}    \Phi: \C[[(\Z_{\ge 0})^{P}]] & \to \C[[(\Z_{\ge 0})^{P}+K_+]] \quad \text{which sends}\\
\nonumber    \nov_{\vec{p}} &\mapsto \nov_{\vec{p}} \cdot \phi_{\vec{p}}(\nov).
\end{align}
It is evident from the definition that this map sends $\C[[K_{\ge 0}]] \to \C[[K_+]]$. 
We define $\C[[K_+]]$ to be the ring of functions on the `simplified K\"ahler moduli space', and $\C[[K_{\ge 0}]]$ the ring of functions on the `simplified complex moduli space'. The restriction of $\Phi$ to the latter is the mirror map, whose computation is outlined in \cite[Section 6.3.4]{coxkatz} (see Appendix \ref{sec:comp_mm}). 
The conjecture is then equivalent to saying that this map has integer Taylor coefficients, i.e., it sends $\Z[[K_{\ge 0}]] \mapsto \Z[[K_+]]$. 
\end{rmk}

Certain cases of Conjecture \ref{conj:intmap} are covered by the works of Krattenthaler--Rivoal \cite{Krattenthaler2010} and Delaygue \cite{Delaygue2013}, but these works assume that the monoid $K_{\ge 0}$ is isomorphic to $(\Z_{\ge 0})^m$ for some $m$, and that $K_{\vec{p}} \subset K_{\ge 0}$ (so that $\gamma_{\vec{p}} = 0$) for all $p \in P$, both of which are certainly not true in general (see Lemma \ref{lem:mon_sm} below for an example). 
Other cases, in which $\Delta$ is highly symmetric, are covered by the work of Beukers--Vlasenko \cite{BeukersVlasenkoIII}.
To the best of our knowledge, the general case of Conjecture \ref{conj:intmap} is open. 

We prove Conjecture \ref{conj:intmap} in the case of Greene--Plesser mirrors, as a natural byproduct of an arithmetic refinement of homological mirror symmetry, and conditionally on some widely-expected foundational results on pseudoholomorphic curve theory and noncommutative geometry:

\begin{main}\label{thm:int_mm}
    Suppose that $\Delta$ is a simplex, there exists a vector $\lambda$ satisfying the MPCS condition (see Definition \ref{def:mpcs} below for an explanation of this condition; note it is automatic when the rank of $M$ is $\le 4$), the relative Fukaya category satisfies the assumptions of     \cite[Section 4]{Ganatra2015}, and \cite[Conjecture 1.14]{Ganatra2015} holds. Then Conjecture \ref{conj:intmap} holds.
\end{main}

\begin{rmk}
    The assumptions of \cite[Section 4]{Ganatra2015} concern the construction of the relative Fukaya category and its cyclic open--closed map, and their structural properties. The relative Fukaya category has been constructed in our context \cite{perutz2022constructing}; however its cyclic open--closed map has only been constructed and showed to respect pairings under more restrictive hypotheses such as tautological unobstructedness \cite{Ganatra_cyclic_OC}; and it has only been proved to be a morphism of variations of Hodge structures under even more restrictive hypotheses \cite{Hugtenburg2022}.
\end{rmk}

\begin{rmk}
    Kontsevich has informed us of an alternative approach to Conjecture \ref{conj:intmap}, making use of the integrality of the coefficients in scattering diagrams in the Gross--Siebert approach to mirror symmetry for Batyrev mirror pairs \cite{KontsevichSoibelmanAffine,GrossSiebertI,GrossSiebertII,Grosstoricdeg}.
\end{rmk}

\begin{example}\label{eg:hyp_in_pn}
    Let $\Delta \subset \R^{n-1}$ be the simplex with vertices $\vec{e}_i$ (the standard basis vectors) and $-\sum \vec{e}_i$. 
    This corresponds to the smooth degree-$n$ hypersurface in $\mathbb{CP}^{n-1}$ which we considered above. 
    In this case $K_{\ge 0}$ is generated by $(1,1,\ldots,1)$, so $\Z[[K_{\ge 0}]] = \Z[[T]]$ where $T = \nov_1\ldots \nov_n$, $\phi_i(T)$ are all equal to $\phi(T)$, and Conjecture \ref{conj:intmap} says that $\phi(T)^n \in \Z[[T]]$. 
    As $\Delta$ is a simplex, our Theorem \ref{thm:int_mm} gives a new proof of the integrality of $\phi(T)^n$, independent of those in \cite{Lian1998,zudilin2002,Lian-Yau-nthroot,Krattenthaler2010}. 
    We do not prove integrality of $\phi(T)$ in this case, although see Remark \ref{rmk:nthroot}. 
\end{example}

\begin{example}\label{eg:mirr_hyp}
    Our Theorem \ref{thm:int_mm} also proves Conjecture \ref{conj:intmap} in the case that $\Delta \subset \R^{n-1}$ is the dual of the reflexive polytope from Example \ref{eg:hyp_in_pn}, corresponding to the `mirror quartic', `mirror quintic', et cetera. 
    In this case, the rank of $K$ is $\binom{2n-1}{n-1} - n^2$, so $19$ for the mirror quartic, $101$ for the mirror quintic, and so on. 
    When $n \ge 3$, the monoid $K_{\ge 0}$ is not isomorphic to $\Z_{\ge 0}^{\mathrm{rk} K}$ by Lemma \ref{lem:mon_sm} below, and furthermore there exist $\vec{p}$ such that $K_{\vec{p}}$ is not contained in $K_{\ge 0}$, so these cases of Theorem \ref{thm:int_mm} are not covered by the existing literature.
\end{example}

\begin{lem}\label{lem:mon_sm}
    Let $\Delta \subset \R^{n-1}$ be the reflexive simplex from Example \ref{eg:mirr_hyp}, $n \ge 3$.
    Then the monoid $K_{\ge 0}$ is not isomorphic to $(\Z_{\ge 0})^{\mathrm{rk} K}$, and furthermore there exist $\vec{p}$ such that $K_{\vec{p}}$ is not contained in $K_{\ge 0}$.
\end{lem}
\begin{proof}
    We start with a general observation. 
    Let $k_{\vec{p}}$ denote the basis vectors of $\Z^P$. 
    Suppose that $C \subset P$ is a subset, whose convex hull is a simplex with vertices $C$, containing the origin in its interior. 
    Then there exists $v^C = \sum_{\vec{p} \in C} v^C_{\vec{p}} k_{\vec{p}} \in K_{\ge 0}$, unique up to positive scaling. 
    We claim that $v^C$ lies on an extremal ray of $K_{\ge 0}$. 
    Indeed, suppose that there existed $w \in K_\R$ such that $v^C + \epsilon w \in \R_{\ge 0}^P$ for $\epsilon \in \R$ sufficiently small. 
    Clearly we must have $w_{\vec{p}} = 0$ for all $\vec{p} \notin C$; so as $w \in K$ we must have $w$ proportional to $v^C$. 

    Now we consider $\Delta = \{x_i \ge -1 \text{ for all $i$, } \sum x_i \le 1\}$. Consider the case $n=3$. Let $E$ and $F$ be opposite edges of $\Delta$, $e_1,e_2,e_3$ the lattice points in the interior of $E$ (with $e_2$ in the centre), $e_4,e_5,e_6$ the lattice points in the interior of $F$ (with $e_5$ in the centre), $e_7$, \ldots, $e_{22}$ the remaining elements of $P$. 
    Let $k_1,\ldots,k_{22}$ be the corresponding basis vectors of $\Z^P$. 
    Note that 
    $$e_1+e_3 = 2e_2 = -2e_5 = -e_4-e_6.$$
    It follows that the following are elements of $K_{\ge 0}$:
    $$a=k_2+k_5;\quad b=k_1+k_3+k_4+k_6;\quad c=k_1+k_3+2k_5;\quad d=2k_2+k_4+k_6.$$
    These all lie on extremal rays of $K_{\ge 0}$ by the above argument.  

    We now observe that $a$, $b$, $c$, and $d$ lie on distinct extremal rays of $K_{\ge 0}$, however they are linearly dependent, as $2a+b=c+d$. 
    This is impossible, if $K_{\ge 0} \cong \Z_{\ge 0}^{19}$. 

    We also note that $e_1+e_3-2e_2 \in K_2 \setminus K_{\ge 0}$, so $K_2$ is not contained in $K_{\ge 0}$.

    For $n>3$, we observe that the intersection of $\Delta$ with the plane $x_4=x_5=\ldots=x_{n-1}=0$ is the 3-dimensional version; so we can carry out the same argument within that slice of $\Delta$. 
\end{proof}

\subsection{Greene--Plesser data}
\label{subsec:tordata}

We recall the Greene--Plesser mirror construction \cite{Greene1990}, using the language of \cite{Batyrev1993}.

Let $M$ be a lattice of rank $n-1$ (by which we mean an abelian group isomorphic to $\Z^{n-1}$), and let us denote $M_\R := M \otimes_\Z \R$. 
Let $\Delta \subset M_\R$ be a reflexive polytope which is a simplex containing the origin as its unique interior lattice point.

Let $\Sigma'$ denote the complete fan in $M_\R$ whose rays point along the vertices of $\Delta$. 
Let $\lambda$ be an element of $(\R_{>0})^P$, where $P \subset \partial \Delta \cap M$ is the set of lattice points on the boundary of $\Delta$ which are not contained in the interior of a codimension-$1$ face. 
Define $\psi_\lambda:M_\R \to \R$ to be the smallest convex piecewise-linear function such that $\psi_\lambda(\vec{p}) \ge -\lambda_{\vec{p}}$ for all $\vec{p} \in P$. 
The decomposition of $M_\R$ into domains of linearity of $\psi_\lambda$ defines a fan $\Sigma_\lambda$.

\begin{defn}[Definitions 1.5 and 1.8 of \cite{sheridan2021homological}]
\label{def:mpcs}
    We say that $\lambda$ satisfies the `MPCP condition' if $\Sigma_\lambda$ is a simplicial refinement of $\Sigma'$. We say that it satisfies the `MPCS condition' if furthermore all cones of $\Sigma_\lambda$ which do not intersect the interior of a top-dimensional cone of $\Sigma'$ are smooth. 
\end{defn}

We note that for any $\Delta$ there exists a $\lambda$ satisfying the MPCP condition, by \cite{OdaPark}; and the MPCP and MPCS conditions are equivalent when $n \le 5$. 
    
The data on which our construction of a Greene--Plesser mirror pair depends are the lattice $M$, the reflexive simplex $\Delta$, and a vector $\lambda$ satisfying the MPCS condition. It will be convenient for the following discussion to explain how this data is equivalent to a choice of toric data as in \cite[Section 1.2]{sheridan2021homological}. We must produce a finite set $I$; positive integers $\{d_i\}_{i \in I}$ satisfying $\sum_i \frac{1}{d_i} = 1$; a sublattice $\ol{M} \subset \Z^I$ containing $\vec{e}_I := \sum_{i \in I} \vec{e}_i$ and all $d_i \vec{e}_i$ (where $\vec{e}_i$ is the $i$th basis vector of $\Z^I$), and such that $d|\langle \vec{q}, \vec{m}\rangle$ for all $\vec{m} \in \ol{M}$ where $d=\lcm(d_i)$ and $q_i = d/d_i$; and a vector $\lambda$ satisfying the MPCS condition. 

Let $N$ be the dual lattice of $M$, and $\nabla \subset N_\R$ the dual of $\Delta$. 
Let $\{\vec{v}_i\}_{i\in I}$ be the vertices of $\Delta$, and $\{\vec{w}_i\}_{i\in I}$ the corresponding vertices of $\nabla$.
Namely, $w_i$ is the unique vertex of $\nabla$ that does not lie on the facet dual to $v_i$.
Then as a consequence of reflexivity, we have
$$\langle \vec{v}_i,\vec{w}_j\rangle +1 = \delta_{ij} d_i$$
for some positive integers $d_i$. 
We set $d = \lcm(d_i)$, and $q_i = d/d_i$. We observe that
$$\left\langle \sum_i q_i \vec{w}_i,\vec{v}_j \right\rangle = q_jd_j - \sum_i q_i = d - \sum_i q_i.$$
As the convex span of the $\vec{v}_j$ is all of $M_\R$, we must have $\sum_i q_i \vec{w}_i = 0$ and $d=\sum_i q_i$. 
It follows that $\sum_i \frac{1}{d_i} = 1$ as required.

We now set $\ol{M} \subset \Z^I$ to be the image of the embedding
\begin{align*}
    \iota: \Z \oplus M & \to \Z^I\\
    \iota(k,m) &:= k \vec{e}_I + \sum_{i \in I} \langle \vec{w}_i,m\rangle \vec{e}_i.
\end{align*}
Note that the image contains $\iota(1,0) = \vec{e}_I$, and $\iota(1,\vec{v}_i) = d_i \vec{e}_i$, as required. 
Furthermore, we have that
\begin{align*}
    \langle \vec{q},\iota(k,m) \rangle &= k\sum_{i \in I} q_i + \sum_{i \in I} \langle \vec{w}_i,m\rangle q_i \\
    &= kd + \left\langle \sum_{i \in I} q_i \vec{w}_i,m\right\rangle \\
    &=kd
\end{align*}
is divisible by $d$. 
We now observe that our $\iota(\{1\} \times \Delta)$ corresponds to $\ol{\Delta}$ from \cite[Section 1.2]{sheridan2021homological}, and our $P$ corresponds to $\Xi_0$. 
Under this correspondence, it is clear that $\lambda$ corresponds to a vector satisfying the MPCS condition.

\subsection{\texorpdfstring{$A$}{A}-side construction}

Let $Y'$ denote the toric variety corresponding to the fan $\Sigma'$. 
Recall that the lattice points $\vec{q} \in \nabla \cap N$ correspond to sections $z^{\vec{q}}$ of a line bundle on $Y'$; we define 
$$X' := \left\{ \sum_{\vec{q} \in \mathrm{Vert}(\nabla)} z^{\vec{q}} = 0\right\} \subset Y'.$$
The refinement $\Sigma_\lambda$ of $\Sigma'$ determines a toric resolution $Y \to Y'$, and we define $X \subset Y$ to be the proper transform of $X'$. 
Define $D \subset X$ to be the intersection of $X$ with the toric boundary divisor $\partial Y$. 
The MPCS condition ensures that $X$ is smooth and $D$ is simple normal-crossings. 
 
We define $\fuk(X,D)$ to be the ambient relative Fukaya category of $(X,D) \subset (Y,\partial Y)$ constructed in \cite{perutz2022constructing}. 
This depends on a choice of a certain $amb$-nice cone $\Nef$ (\cite[Definition 3.39]{Sheridan2017}), but we will omit it from the notation. 
The ambient relative Fukaya category is a $\Z$-graded curved filtered $A_\infty$ category defined over the graded ring $R_A = \Z[[\NE_A]]$ where $\NE_A \subset \Z^{P}$ is dual to the cone $\Nef$.\footnote{Note that $\fuk(X,D)$ was defined over $\C[[\NE_A]]$ in \cite{sheridan2021homological}, but the construction of \cite{perutz2022constructing} works over $\Z$.} 

\begin{lem}\label{lem:K+inNEA}
    We have $K_+ + (\Z_{\ge 0})^P \subset \NE_A$. 
    In particular, $K_+ + (\Z_{\ge 0})^P$ is strongly convex.
\end{lem}
\begin{proof}
    It is clear that $(\Z_{\ge 0})^P \subset \NE_A$; thus it suffices to show that $K_{\vec{p}} \subset \NE_A$ for all $\vec{p}$. 
    Let $u \in K_{\vec{p}}$. 
    As $\Nef$ is $amb$-nice, and the component $D_{\vec{p}}$ is non-empty, for every $\lambda \in \Nef$ there exists $\lambda^{(\vec{p})} \in \Nef \cap \R^{P \setminus \{\vec{p}\}}$ such that $(\lambda - \lambda^{(\vec{p})}) \cdot u = 0$ for all $u \in K$. 
    It is then clear that $\lambda^{(\vec{p})} \cdot u \ge 0$, as the only negative entry of $u$ is the $\vec{p}$th one, while all entries of $\lambda^{(\vec{p})}$ are non-negative and the $\vec{p}$th entry is $0$. 
    Thus $\lambda \cdot u = \lambda^{(\vec{p})} \cdot u \ge 0$ for all $\lambda \in \Nef$, so $u \in \NE_A$. 
    It follows that $K_{\vec{p}} \subset \NE_A$ for all $\vec{p}$, hence $K_+ \subset \NE_A$ as required.
\end{proof}

Let $\Bbbk$ be a field. 
The vector $\lambda$ determines the cohomology class of an ambient relative K\"ahler form $[\omega;\theta] \in H^2(X,X \setminus D;\R)$ and hence an algebra homomorphism
$d(\lambda)^*: R_A \to \Lambda_{\Bbbk,Q}$ sending $\nov_p \mapsto T^{\lambda_p}$ for all $p \in P$, where
\begin{equation} \Lambda_{\Bbbk,Q} := \left\{ \sum_{j=0}^\infty c_j \cdot 
T^{\lambda_j}: c_j \in \Bbbk, \lambda_j \in Q, \lim_{j \to \infty} \lambda_j = +\infty\right\}\end{equation}
is a Novikov field over $\Bbbk$ and $Q \subset \R$ is a subgroup containing all $\lambda_p$. 
We define 
$$\fuk(X,\omega;\Lambda_{\Bbbk,Q}) := \fuk(X,D)_{d(\lambda)}$$ 
to be the fibre of the relative Fukaya category over the $\Lambda_{\Bbbk,Q}$-point $d(\lambda)$. 
We define $\fuk(X,\omega;\Lambda_{\Bbbk,Q})^\bc$ to be the category of bounding cochains in this category. 
It is a $\Z$-graded, $\Lambda_{\Bbbk,Q}$-linear, uncurved $A_\infty$ category.

\begin{rmk}
    One could envision a more general definition of $\fuk(X,\omega;\Lambda_{\Bbbk,Q})$ which includes Lagrangians which are not exact in the complement of $D$ as objects. 
    As long as this definition satisfies the analogues of the assumptions enumerated in \cite[Section 2.5]{sheridan2021homological}, we would expect all of our results to hold for it. 
    Such an enlarged Fukaya category has been defined when $\Q \subset \Bbbk$ and $Q = \R$, see \cite{fukaya2017unobstructed}. 
    For any $\Q \subset \Bbbk$ and $Q \subset \Q$ we can define a version over $\Lambda_{\Bbbk,Q}$ by restricting objects to be rational Lagrangians with holonomy in $Q$, see \cite{fukaya-galois}. 
    It is less clear how to work over arbitrary $\Bbbk$. 
    One notable exception is when $X$ is a K3 surface, see \cite[Section 8c]{Seidel2003} (although the construction is written for $\Bbbk = \C$, it works for arbitrary $\Bbbk$ by inspection). 
    In general, one may speculate that the techniques of \cite{Bai-Xu} could be brought to bear, but in our case the space is Calabi--Yau, so one could hope to give a simpler definition using the techniques of \cite{Hofer-Salamon} to rule out the appearance of sphere bubbles (the only source of non-trivial automorphism groups in our moduli spaces and hence of denominators in our disc counts).
\end{rmk}

\subsection{\texorpdfstring{$B$}{B}-side construction}

In this section, all varieties and stacks are defined over the field $\Lambda_{\Bbbk,Q}$ (which may have finite characteristic and not be algebraically closed). Given a finitely generated abelian group $G$, we let $\underline{Hom}(G,\mathbb{G}_m)$ denote its Cartier dual over $\Lambda_{\Bbbk,Q}$ (here $G$ is regarded as a constant $\Lambda_{\Bbbk,Q}$ group scheme). 
For any $b \in \Lambda_{\Bbbk,Q}^P$, we define 
$$W_b = -z^{\iota(1,0)} + \sum_{\vec{p} \in P} b_{\vec{p}} \cdot z^{\iota(1,\vec{p})} \in \Lambda_{\Bbbk,Q}[z_i]_{i \in I}.$$
Giving $z_i$ degree $q_i$, $W_b$ is a weighted homogeneous polynomial of degree $d$.
Thus, its vanishing locus defines a hypersurface in the weighted projective stack $\mathbb{WP}(\vec{q})$. 
The natural action of $\underline{Hom}(\Z^I,\mathbb{G}_m) \cong \mathbb{G}_m^{I}$ on $\mathbb{A}^I \setminus 0$ descends to an action of $\underline{Hom}(\ker(\vec{q}),\mathbb{G}_m)$ on $\mathbb{WP}(\vec{q})$. 
We consider the subgroup $\Gamma := \underline{Hom}(\ker(\vec{q})/\iota(0 \oplus M),\mathbb{G}_m)$, and define the toric stack $\check{Y}:=[\mathbb{WP}(\vec{q})/\Gamma]$. The action of $\Gamma$ preserves the vanishing locus of $W_b$, and we define $\check{X}_b \subset \check{Y}$ to be the quotient stack. 
We define $D^bCoh(\check{X}_b)$ to be its derived category.

\subsection{Arithmetic homological mirror symmetry}

Homological mirror symmetry was originally envisioned as an equivalence of $\C$-linear categories. 
Lekili--Perutz proposed that it ought to admit a refinement over $\Z$ called `arithmetic HMS', and proved such a result for an elliptic curve \cite{Lekili-Perutz}. 
Other arithmetic HMS results include \cite{Lekili-Polishchuk-2017,Lekili-Polishchuk-2020,Lekili-Polishchuk-2023,Seidel-formal-groups,Cho-Amorim,Smith2022}. 
Progress for higher-dimensional compact Calabi--Yaus was hampered by the fact that in general, the Fukaya category is only defined over a field of characteristic zero \cite{fooo,fukaya2017unobstructed}. 
This impediment was removed in \cite{perutz2022constructing}, where the Fukaya category of a Calabi--Yau variety relative to a divisor was defined over a ring of formal power series with integer coefficients. 
We prove arithmetic HMS for Greene--Plesser mirrors, away from a finite set of characteristics for which the $B$-side variety has bad reduction:

\begin{defn}
    Observe that the fan $\Sigma_\lambda$ induces a decomposition of $\Delta$ into simplices with vertices on $P \cup \{0\}$. 
    We define $\lcm(\lambda)$ to be the least common multiple of the affine volumes of these simplices where the affine volume is normalized so that the minimal possible non-zero volume is $1$.
\end{defn}

\begin{main}\label{thm:main}
    Suppose that $\mathrm{char}(\Bbbk) \nmid \lcm(\lambda)$, and furthermore, the analogues of the assumptions enumerated in \cite[Section 2.5]{sheridan2021homological} are satisfied by $\fuk(X,D)$.  
    Then there exist $\psi_{\vec{p}} \in R_A$ satisfying $\psi_{\vec{p}} = \pm 1 \operatorname{mod} \fm$ (where $\fm \subset R_A$ is the ideal corresponding to the vertex of $\NE_A$), so that if we set $b_{\vec{p}} = d(\lambda)^*(\nov_{\vec{p}} \cdot \psi_{\vec{p}})$ for $\vec{p} \in P$, we have a quasi-equivalence
        $$\mathrm{Perf} \fuk(X,\omega;\Lambda_{\Bbbk,Q})^\bc \simeq D^bCoh(\check{X}_b).$$
\end{main}

    The assumptions of \cite[Section 2.5]{sheridan2021homological} are satisfied in the case that $X$ is a K3 surface, for $\Bbbk$ an arbitrary field and $\Q \subset Q$, see \cite[Remarks 2.6 and 2.7]{sheridan2021homological}. 
    Thus, we may remove this hypothesis from Theorem \ref{thm:main} when $n=4$.
    We can also enlarge the definition of $\fuk(X,\omega;\Lambda_{\Bbbk,Q})$ to include arbitrary rational Lagrangians in $X$ in this case. 

   Theorem \ref{thm:main} was proved when $\Bbbk = \C$ in \cite{sheridan2021homological}, except it was only shown that the formal power series $\psi_{\vec{p}}$ had coefficients in $\C$. On the other hand, [op. cit.] also treated certain `generalized' Greene--Plesser mirror pairs. We expect our methods could be used to prove analogues of Theorems \ref{thm:int_mm} and \ref{thm:main} for these generalized Greene--Plesser mirror pairs.

We prove Theorem \ref{thm:main} before Theorem \ref{thm:int_mm}. 
The proof of Theorem \ref{thm:int_mm} consists in showing that
    \begin{equation}
\label{eq:prodphi_prodpsi}        \prod_{\vec{p} \in P} \phi_{\vec{p}}^{u_{\vec{p}}} = \prod_{\vec{p} \in P} \psi_{\vec{p}}^{u_{\vec{p}}}
    \end{equation}
for all $u \in K$, where $\psi_{\vec{p}}$ are the power series appearing in the statement of Theorem \ref{thm:main}. 
The result follows, as the $\psi_{\vec{p}} \in R_A := \Z[[\NE_A]]$ have integer Taylor coefficients by construction.

\begin{rmk}\label{rmk:nthroot}
We continue our discussion of the case of Calabi--Yau hypersurfaces in projective space, from Example \ref{eg:hyp_in_pn}. 
If we could arrange for all of our constructions to respect the action of $Sym(n)$ permuting the homogeneous coordinates on both sides, then we could conclude that the $\psi_{\vec{p}}$ were all equal, which would imply that moreover $\phi(T) = \psi_{\vec{p}}(T) \in \Z[[T]]$ as proved in \cite{Lian-Yau-nthroot} for $n$ prime and in \cite{Krattenthaler2010} in general. 
However, we do not pursue this idea here.
\end{rmk}

\subsection{Overview of proofs}

Let us give a quick overview of the proofs of our main theorems to illustrate the key concepts. As mentioned above, homological mirror symmetry when $\Bbbk = \C$ was proved for (generalized) Greene--Plesser mirror pairs in \cite{sheridan2021homological}. 
The proof starts by introducing a subcategory $\tilde{\mathbb{A}}' \subset \fuk(X,D)$, defined over $\C[[\NE_A]]$, and a corresponding subcategory $\tilde{\mathbb{B}}$ of a certain category of equivariant graded matrix factorizations, $\grmf_\Gamma(S,W)$, defined over $\C[[(\Z_{\ge 0})^P]]$. 
One identifies $\tilde{\mathbb{A}}'_0 \simeq \tilde{\mathbb{B}}_0$ using \cite{Sheridan2011} and then uses the versality result of \cite{Sheridan2017} to construct a mirror map $\Psi^*:\C[[(\Z_{\ge 0})^P]] \to \C[[\NE_A]]$ and a filtered quasi-isomorphism $\tilde{\mathbb{A}}' \simeq \Psi^*\tilde{\mathbb{B}}$. 

The key new technical ingredients in the present work are the construction of the relative Fukaya category over $\Z[[\NE_A]]$ from \cite{perutz2022constructing} and a new versality argument which works over $\Z[[\NE_A]]$ (see Appendix \ref{sec:vers}). 
The versality result in \cite{Sheridan2017} required a field of characteristic zero to translate the deformation theory of $A_\infty$ algebras into the deformation theory of $L_\infty$ algebras, which is only well-behaved in characteristic zero. 
The versality result given here deals directly with $A_\infty$ structures and works over $\Z$. 
We also comment on a further modification of the versality result made here: the argument in \cite{sheridan2021homological} ruled out certain deformations using the existence of a certain (signed) group action, which would have necessitated inverting the order of the group in the coefficient ring even when working directly with the $A_\infty$ structure. 
The argument given here is different in nature: it rules out these deformations by showing that they are obstructed. 

With these ingredients in place, it remains to run through the argument of \cite{sheridan2021homological} and make the necessary adjustments so that it works over $\Z$. 
The upshot is a filtered quasi-isomorphism $\tilde{\mathbb{A}}' \simeq \Psi^*\tilde{\mathbb{B}}$, where $\Psi^*:R_B \to R_A$ has integer Taylor coefficients by construction. 
 Passing to bounding cochains, one obtains a subcategory $\tilde{\mathbb{A}} \subset \fuk(X,D)^\bc$ with a quasi-isomorphism $\tilde{\mathbb{A}} \simeq \Psi^*\tilde{\mathbb{B}}$. 
One then passes to the fibre over $d(\lambda)$, to get a quasi-isomorphism $\tilde{\mathbb{A}}_{d(\lambda)} \simeq \tilde{\mathbb{B}}_b$, where $b=\Psi(d(\lambda))$. 
One uses a tropical regularity criterion (see Section \ref{sec:tropreg}) to show that $W_b$ has an isolated singularity at the origin (this is where the condition $\mathrm{char}(\Bbbk)\nmid \lcm(\lambda)$ is needed). 
This, together with the fact that we are now working over a field (namely $\Lambda_{\Bbbk,Q}$), allows us to apply Orlov's theorem \cite{Orlov2009} to show that $\grmf_\Gamma(S,W)_b$ embeds in $D^bCoh(\check{X}_b)$, and furthermore, that the image of $\tilde{\mathbb{B}}_b$ split-generates. 
We then apply the automatic split-generation criterion of \cite{Ganatra2016, Sanda2021} to conclude that $\tilde{\mathbb{A}}_{d(\lambda)}$ also split-generates, concluding the proof of Theorem \ref{thm:main}.

Under the assumptions of Theorem \ref{thm:int_mm}, \cite{Ganatra2015} proves that homological mirror symmetry (the version with $\Bbbk = \C$) implies Hodge-theoretic mirror symmetry.
This is used in \cite[Appendix C]{Sheridan2017} to show that the restriction of $\Psi^*$ to $\C[[K_{\ge 0}]]$ is uniquely characterized by the fact that it preserves flat coordinates. 
It is computed in \cite[Section 6.3.4]{coxkatz} that the unique mirror map preserving flat coordinates is the restriction of $\Phi$ to $\C[[K_{\ge 0}]]$, where $\Phi$ is as in \eqref{eq:mm_formula}. 
Thus we have $\Phi|_{\C[[K_{\ge 0}]]} = \Psi^*|_{\C[[K_{\ge 0}]]}$; as $\Psi^*$ has integer Taylor coefficients by construction, this concludes the proof of Theorem \ref{thm:int_mm}. 

\paragraph{Acknowledgments}
We are very grateful to Sophie Bleau for developing computer code to check cases of Conjecture \ref{conj:intmap}.
We are very grateful to Masha Vlasenko for pointing out and correcting a serious mistake we had made in the formula for the mirror map in Conjecture \ref{conj:intmap}, among other helpful correspondence.  
We would also like to thank Mauricio Romo for pointing out \cite{Batyrev1995} and Duco van Straten for helpful correspondence concerning the integrality of mirror maps problem.
S.G. was supported by NSF grant CAREER DMS-2048055 and a Simons Fellowship (award number 1031288).
A.H. was supported by the Simons Foundation (Grant Number
814268 via the Mathematical Sciences Research Institute, MSRI).
J.H. is supported by an EPSRC Postdoctoral Fellowship (project reference: EP/V049097/1).
D.P. received partial funding from NSF grant DMS-2306204. 
N.S. is supported by a Royal Society University Research Fellowship, an ERC Starting Grant (award number 850713 -- HMS), the Simons
Collaboration on Homological Mirror Symmetry (award number 652236), the Leverhulme Prize, and a Simons Investigator award (award number 929034).

\section{Computations in the Fukaya category}

We now set up the versality argument on the symplectic side. 
We consider the immersed Lagrangian sphere $L$ which was constructed in \cite{Sheridan2011}. 
It sits inside the pair of pants $\tilde{X}' \setminus D'$, of which $X \setminus D$ is a cover, as we review in Section \ref{subsec:branchedCover}. 
We compute the endomorphism algebra $\cA_0$ of $L$ in Section \ref{subsec:endomorphism}, its Hochschild cohomology in Section \ref{subsec:hochschild}, and the first-order deformation classes associated to the components of $D$ in Section \ref{subsec:ambient}. These sections closely follow \cite{sheridan2021homological,Sheridan2015,Sheridan2011}, making the necessary adaptations to work over $\Z$ coefficients. 
In Section \ref{subsec:verifyVersality}, we show that the relative Fukaya category satisfies the versality hypothesis of Proposition \ref{prop:vers}, which is different from that used in \cite{sheridan2021homological}.

\subsection{Branched cover}
\label{subsec:branchedCover}
Recall (from \cite[Section 1.3]{sheridan2021homological}) that we have a branched covering of sub-snc pairs $\phi:(X,D) \to (\tilde{X}',\tilde{D}')$ in the sense of \cite[Section 4.9]{Sheridan2017}, where 
\begin{align*}
    \tilde{X}' &= \left\{\sum_{i \in I} z_i = 0\right\} \subset \CP^{|I|-1},
\end{align*} 
and $\tilde{D}' \subset \tilde{X}'$ is the intersection with the toric boundary of the projective space.

We denote the corresponding ambient grading data by 
\begin{align*}
    \mathbb{G} := \mathbb{G}_{amb}(X \setminus D) & \cong \Z \oplus M,\\
    \tilde{\mathbb{G}} := \mathbb{G}_{amb}(\tilde{X}' \setminus \tilde{D}') & \cong \Z \oplus \Z^I/\langle(2(1-|I|),e_{I})\rangle,
\end{align*} 
where $M = \ol{M}/\langle e_{I_j}\rangle$.  
The morphism of ambient grading data induced by $\phi$ is denoted by
\begin{align*}
    \mathbf{p}: \mathbb{G} & \to \tilde{\mathbb{G}},\\
    \mathbf{p}(k,\vec{m}) & = \left(k-2\sum_{i \in I} \langle \vec{w}_i,\vec{m}\rangle,\iota(0,\vec{m})\right).
\end{align*}

\subsection{Endomorphism algebra of the immersed Lagrangian}
\label{subsec:endomorphism}
An exact immersed Lagrangian sphere $L \looparrowright \tilde{X}' \setminus \tilde{D}'$, together with anchoring and Pin structure, was constructed in \cite{Sheridan2011}. 
The endomorphism algebra $\cA_0$ of $L$ was computed up to quasi-isomorphism over a field, but the computation works over $\Z$ as we now explain. 

First, we compute the cohomological endomorphism algebra $H \cA_0$ of $L$ in $H\fuk(\tilde{X}' \setminus \tilde{D}';\Z)$ (an associative $\Z$-algebra). 
It is isomorphic to $\Z[\theta_i]_{i \in I}$, where $\theta_i$ are anti-commuting variables of degree $(-1,\vec{e}_i) \in \tilde{\mathbb{G}}$ by \cite[Theorem 5.12]{Sheridan2011}. 
The computation is written for $\C$ but manifestly works over $\Z$.

\begin{lem}\label{lem:hkr}
    We define a $\tilde{\G} \oplus \Z$-graded Gerstenhaber algebra $\Z[z_i,\theta_i]_{i \in I}$, where $z_i$ are commuting variables of degree $((2,-\vec{e}_i),1)$, and $\theta_i$ are anti-commuting variables of degree $((-1,\vec{e}_i),0)$, and the Lie bracket is the Schouten bracket. 
    Then the HKR map
    \begin{align} \label{eq:HKRoverZ}
        \Xi: CC^*\left(\Z[\theta_i]_{i \in I}\right) &\to \Z[z_i,\theta_i]_{i \in I}\\
        \Xi(\alpha) &= \sum_{i=0}^\infty \alpha^i(\mathbf{z},\ldots,\mathbf{z}), \nonumber
    \end{align}
    where $\mathbf{z} = \sum_{i \in I}z_i \theta_i$, induces a $\tilde{\mathbb{G}} \oplus \Z$-graded isomorphism of Gerstenhaber algebras. 
\end{lem}
\begin{proof}
It is standard that the analogous map $\Xi_\mathbb{C}$ defined over $\mathbb{C}$ (or any field of characteristic zero) is an isomorphism of Gerstenhaber algebras, see \cite[Section 4b]{Seidel2003} and \cite[Proposition 5.4.6]{Loday:cyclic}. We now explain why this also holds over $\mathbb{Z}.$ 

\emph{Injectivity:} We first show that  \eqref{eq:HKRoverZ} is injective.  Let $S = \Z[\theta_i]_{i \in I}$ be the exterior algebra over $\Z$, and $S^\C = S \otimes_\Z \C$.  Because $S$ is finitely generated and free as a $\mathbb{Z}$-module, we have a base change isomorphism: \begin{align*} HH^*(S^{\mathbb{C}}|\mathbb{C}) \cong HH^*(S|\mathbb{Z})\otimes_\mathbb{Z} \mathbb{C}. \end{align*}

Moreover,  the map $\Xi_\mathbb{C}$ is obtained from $\Xi$ by base-change.  Therefore, since we know $\Xi_\mathbb{C}$ is an isomorphism,  it suffices to prove that $HH^*( S|\mathbb{Z})$ is torsion-free.   If $ S= \mathbb{Z}[\theta]$ is an exterior algebra in one variable $\theta$,  this follows immediately from the fact that the Hochschild differential on the normalized Hochschild complex vanishes identically.  In general,  \begin{align*}  S \cong \bigotimes_i \mathbb{Z}[\theta_i]  \end{align*} as graded algebras.  Then,  the tensor product of normalized bar complexes for $\mathbb{Z}[\theta_i]$ defines a resolution of the diagonal bimodule over $ S\otimes_\mathbb{Z}  S^{op}$.  This implies that $HH^*( S|\mathbb{Z})$ satisfies the K\"unneth formula: \begin{align*} HH^m( S|\mathbb{Z}) \cong \bigoplus_{m_1,\cdots,m_n, \sum m_i=m}  \bigotimes_i HH^{m_i}(\mathbb{Z}[\theta_i]|\mathbb{Z}),  \end{align*}  implying the claim in general.     \vskip  5 pt

\emph{Surjectivity:} Because $\Xi$ is injective and $\Xi_\mathbb{C}$ is a map of Gerstenhaber algebras, $\Xi$ is  a map of Gerstenhaber algebras.  By considering Hochschild cochains of length $\leq 1$,  it is immediate that $\theta_i$ and $z_i$ are in the image of $\Xi$ for all $i \in I$. Because these generate $\mathbb{Z}[z_i,\theta_i]_{i \in I}$ as an algebra,  the map  \eqref{eq:HKRoverZ} is also surjective. 
\end{proof}

\begin{cor}\label{cor:hhdef}
    Define $W_0 = -z^{\vec{e}_I} \in \Z[z_i]_{i \in I} = H\cA_0$. 
    Then we have 
    $$\HH^2( H\cA_0)^s = \begin{cases}
                            0 & \text{if $s > 2$, $s \neq |I|$}\\
                            \Z \cdot W_0 & \text{if $s=|I|$.}
    \end{cases}     $$
\end{cor}
\begin{proof}
    This is \cite[Lemma 2.96]{Sheridan2015}, which is written over $\C$, but given Lemma \ref{lem:hkr}, manifestly works over $\Z$.
\end{proof}

The $A_\infty$ structure $\mu_0$ on $\cA_0$ satisfies $\mu_0^s = 0$ unless $s$ is congruent to $2$ mod $|I|-2$ by \cite[Lemma 5.9]{Sheridan2011}. 
It follows that $\mu_0^{|I|}$ is closed under the Hochschild differential. 
It is shown in \cite[Proposition 5.13]{Sheridan2011} that $\Xi(\mu_0^{|I|}) = \pm W_0$. 
By Corollary \ref{cor:hhdef} and \cite[Lemma 3.2]{Seidel2003}, the class $\Xi(\mu_0^{|I|})$ determines the $A_\infty$ structure up to formal diffeomorphism; once again the proof of the latter lemma goes through over $\Z$.

\subsection{Hochschild cohomology of the \texorpdfstring{$A_\infty$}{A-infinity} algebra}

\label{subsec:hochschild}

We compute the Hochschild cohomology of $\cA_0$.  We consider the length filtration on $CC^*(\cA_0)$. 
If $\phi_i$ are Hochschild cochains of length $\ell_i$ for $i=1,2$, then their Gerstenhaber bracket has length $\ell_1 + \ell_2 -1$. 
It follows that the spectral sequence $(E_r^{y,s},d_r^{y,s})$ induced by the length filtration on $CC^*(\cA_0)$ is a spectral sequence of Lie algebras. 

We have $E_2^{y,s} = \HH^y( H\cA_0)^s$. 
The fact that $\mu^s = 0$ for $3 \le s \le |I|-1$ implies that $E_{|I|-1}^{ys} = E_2^{ys}$. 
The differential $d_{|I|-1}$ is equal to bracketing with the element $[\mu_0^{|I|}]$, which corresponds to Schouten bracket with $\pm W_0$ under the HKR isomorphism. 
The Schouten bracket with $W_0$ defines a differential on $\Z[z_i,\theta_i]_{i \in I}$ giving the Koszul complex $K(dW_0)$. 
Thus, we have $E_{|I|}^{y,s} = H(K(dW_0))^{y,s}$.
The cohomology of this Koszul complex was computed in \cite[Lemma 3.5]{sheridan2021homological} (the proof was written over $\C$, but works unchanged over $\Z$):

\begin{lem}
Let $u_i = z_i \theta_i$, and let $\mathcal{J} \subset \Z[z_i,\theta_i]_{i \in I}$ be the subalgebra generated by the elements $z_i$ and $u_i - u_j$.  
Let $\mathcal{I} \subset \mathcal{J}$ be the ideal generated by the elements $z^K \cdot \wedge^{top}(U_K)$ for all subsets $K \subset I$, where $z^K$ means the product of $z_i$ over all $i \in K$, and $U_K$ is the subspace spanned by $u_i - u_j$ for $i,j \notin K$. 
Then we have $H(K(dW_0)) \cong \mathcal{J}/\mathcal{I}$.
\end{lem}

\begin{lem}\label{lem:agradhh}
The associated graded of the length filtration on $\HH^y(\cA_0)$ is given by
$$\mathrm{Gr}_s \HH^y(\cA_0) \cong (\mathcal{J}/\mathcal{I})^{y,s}.$$
\end{lem}
\begin{proof}
By \cite[Theorem 5.5.10]{Weibel-HA} and the fact that the length filtration is complete and bounded above by definition, it suffices to prove that the spectral sequence degenerates at $E_{|I|} \cong \mathcal{J}/\mathcal{I}$.
We claim that $d_r$ vanishes for all $r \ge |I|$ for grading reasons.  

Indeed, let us suppose that there exist $(y_1,s_1), (y_2,s_2) \in Y \oplus \Z$ such that $H(K(dW_0))^{y_i,s_i} \neq 0$ for both $i$, $y_2 = y_1+1$, and $s_2 \ge s_1 + |I|$. 
Then there exist elements $z^{a_i}\theta^{K_i} \in H(K(dW_0))^{y_i,s_i}$. 
This means $(y_i,s_i) = ((2|a_i| - |K_i|,-a_i + \vec{e}_{K_i}),|a_i|)$ (where $|a_i| = a_i \cdot \vec{e}_I$). 
The fact that $y_2 = y_1+1$ implies
$$2|a_2| - |K_2| = 2|a_1| - |K_1| +1 - 2q(1-|I|), \qquad -a_2 + \vec{e}_{K_2} = -a_1+\vec{e}_{K_1} - q \vec{e}_I$$
for some $q \in \Z$. 
The second of these implies
$$-|a_2| + |K_2| = -|a_1| + |K_1| + q|I|;$$
adding this to the first gives
$$|a_2| = |a_1| + 1+q(|I|-2).$$
As $|a_2| \ge |a_1| + |I|$ and $|I| \ge 3$, we must have $q \ge 2$. 
But then 
$$a_2 = a_1 + \vec{e}_{K_2} - \vec{e}_{K_1} + q\vec{e}_I$$
has all entries $\ge 1$, which implies that it is in the ideal generated by $W_0$; thus it is contained in $\mathcal{I}$, so $H(K(dW_0))^{y_2,s_2} = 0$.
\end{proof}

\subsection{Ambient relative Fukaya category}
\label{subsec:ambient}

Let $\hat{\mathbb{A}} \subset \fuk(X,D;\Z)$ denote the subcategory whose objects are lifts of $L$ under the branched covering $\phi$, together with all their shifts. 
We may choose our perturbation data equivariantly for the action of the shifts and covering group, because the former acts trivially and the latter acts freely on the underlying geometric Lagrangians. 
Then we have $\hat{\mathbb{A}} = \mathbf{p}^* \mathbb{A}$ for some $\mathbf{p}_* R_A$-linear $\tilde{\mathbb{G}}$-graded $A_\infty$ category $\mathbb{A}$. 
We denote by $\cA$ the endomorphism algebra of the object $L$ of $\mathbb{A}$. 
Note that all objects of $\mathbb{A}$ are shifts of $L$, so $\cA$ completely determines $\mathbb{A}$ (c.f. \cite[Appendices A and B]{Sheridan2017}). 
We have $\cA \otimes_{R_A} R_A/\fm_A = \cA_0$.

The first-order deformation class of $\cA$ associated to $\vec{p} \in P$ is $\pm z^{\vec{p}}$ by the argument from \cite[Section 3.4]{sheridan2021homological} (which works over $\Z$).

\subsection{Verifying the versality hypotheses}

\label{subsec:verifyVersality}
Let $U  \subset  \Z\langle u_i\rangle_{i \in I}$ be the kernel of $\vec{e}_{I}\cdot(-)$. 
Let 
\begin{align*}
    \mathcal{H}^* &:= \wedge^*(U)/\left(\wedge^{top}(U)\right).
\end{align*}
Note that $\mathcal{H}^i \subset \mathrm{Gr}_i\HH^i(\cA_0)$ can naturally be regarded as a graded subalgebra by Lemma \ref{lem:agradhh}. 

\begin{lem}\label{lem:hh2}
For any $\vec{p} \in P$, let $y_{\vec{p}} = \mathbf{p}(0,\vec{p})$. Then we have:
\begin{itemize}
    \item $\HH^2(\cA_0) = \mathcal{H}^2$;
    \item For any $\vec{p} \in P$, $\HH^{2+y_{\vec{p}}}(\cA_0)$ is a free $\Z$-module of rank one, generated by the class $z^{\vec{p}}$;
    \item If $y \in y(\NE_A) \setminus \left(\{0\} \cup \{y_{\vec{p}}\}_{\vec{p} \in P}\right)$, then $\HH^{2+y}(\cA_0) = 0$.
\end{itemize} 
\end{lem}
\begin{proof}
    Given Lemma \ref{lem:agradhh}, these all follow from \cite[Lemma 3.8]{sheridan2021homological}, if one replaces `$\HH^*$' with `$\mathrm{Gr}_*\HH^*$' everywhere. 
    It follows that for each $y \in y(\NE_A)$, there exists $s(y)$ such that $\mathrm{Gr}_s\HH^{2+y}(\cA_0) = 0$ for $s \neq s(y)$. 
    Hence we have $\mathrm{Gr}_*\HH^{2+y}(\cA_0) \cong \HH^{2+y}(\cA_0)$, which allows us to conclude.
\end{proof}

\begin{lem}\label{lem:obst}
    If $\alpha \in \mathcal{H}^2$ satisfies $[z_i^{d_i},\alpha] = 0$ for all $i \in I$, then $\alpha = 0$.
\end{lem}
\begin{proof}
    By definition of the Schouten bracket, we have 
    \begin{align*}
        [z_i^{d_i},(u_j - u_k) \wedge (u_l - u_m)] & = - \iota_{d_i z_i^{d_i - 1} dz_i}\left((z_j \partial_{z_j} - z_k \partial_{z_k}) \wedge (z_l \partial_{z_l} - z_m \partial_{z_m})\right) \\
        &= d_i z_i^{d_i} \iota_{\vec{e}_i}(\alpha).
    \end{align*}
    Now if $z_i^{d_i} \beta = 0$, for some $\beta \in \mathcal{H}^1$, then we must have $I = \{i,j,k\}$ and $\beta$ must be a multiple of $u_j - u_k$. 
    But in that case $\mathcal{H}^2 = 0$, so $\alpha = 0$.

    Therefore, as $\HH^*(\cA_0)$ is torsion-free by Lemma \ref{lem:agradhh}, we have 
    $$ [z_i^{d_i},(u_j - u_k) \wedge (u_l - u_m)] = 0 \quad \Leftrightarrow \quad \iota_{\vec{e}_i}(\alpha) = 0.$$
    If this is true for all $i$, then clearly $\alpha = 0$.
\end{proof}

\section{Proofs of Theorems \ref{thm:int_mm} and \ref{thm:main}}

\subsection{Minimal model for generator of matrix factorization category}

We consider the $\tilde{\mathbb{G}}$-graded $\mathbf{p}_* R_B$-linear category of matrix factorizations introduced in \cite[Section 4.1]{sheridan2021homological} (restricting to the case $r=1$). 
Let $\mathcal{O}_0$ denote the object of this category introduced in \cite[Section 7.2]{Sheridan2015}, and let $\cB^{dg}$ be its endomorphism algebra. 
We construct a minimal model $\cB$ for $\cB^{dg}$. 
In characteristic zero, the construction was carried out in \cite[Section 7.2]{Sheridan2015}, but the contracting data used there involved arbitrary denominators. 
We now construct contracting data over $\Z$ following the sketch in \cite[Section 5.6]{Dyckerhoff2009}.

\begin{defn}
    Let $i: (C_0,d_0) \hookrightarrow (C_1,d_1)$ be the inclusion of a subcomplex. 
    We define \emph{contracting data} for $i$ to consist of a chain map $p: (C_1,d_1) \to (C_0,d_0)$ and a homomorphism $h:C_1 \to C_1$,
    satisfying
    $$ pi = \id;\quad ip = \id - [d_1,h]; \quad h^2 = 0;\quad hi = 0;\quad ph = 0.$$
    The last three conditions are sometimes called the \emph{side conditions}.
\end{defn}

\begin{lem}\label{lem:comp_cd}
    Let $i_0:(C_0,d_0) \hookrightarrow (C_1,d_1)$ and $i_1:(C_1,d_1) \hookrightarrow (C_2,d_2)$ be inclusions of subcomplexes, admitting contracting data $(p_1,h_1)$ and $(p_2,h_2)$. 
    Then we have contracting data
    $$p = p_1 p_2,\quad h = h_2 + i_2 h_1 p_2$$
    for $i_1i_0$.
\end{lem}
\begin{proof}
    Easy exercise in formula-pushing.
\end{proof}

Let $R$ be a commutative ring, and let $S=R[x_1,\ldots,x_n]$ be the polynomial ring over $R$. 
We consider the ring $\tilde B = S[\theta_i,\partial/\partial \theta_i]_{i=1,\ldots,n}$ equipped with the differential $d_0 = [\delta_0,-]$, where $\delta_0 = \sum_j x_j \partial/\partial \theta_j$. 
Let $i: C \hookrightarrow \tilde B$ denote the inclusion of the subcomplex $C = R[\partial/\partial \theta_i]$, equipped with the $0$ differential.  

\begin{lem}\label{lem:contdata}
    There exist $R$-linear contracting data $(p,h)$ for $i$.
\end{lem}
\begin{proof}
    First, we observe that as a chain complex of $R$-modules, $(\tilde B,d_0)$ is isomorphic to the tensor product of the subcomplex $(S[\theta_i],\delta_0)$ with the free $R$-module $C$. 
    Thus it suffices to construct contracting data for the inclusion $i:(R,0) \hookrightarrow (S[\theta_i],\delta_0)$. 
    By Lemma \ref{lem:comp_cd}, it suffices to construct $R$-linear contracting data for the inclusions
    $$i_j: R[x_i,\theta_i]_{i=1,\ldots,j-1} \hookrightarrow R[x_i,\theta_i]_{i=1,\ldots,j}.$$
    We choose the contracting data $(p_j,h_j)$ to be $R[x_i,\theta_i]_{i=1,\ldots,j-1}$-linear, and satisfy
    \begin{align*}
p_j(x_j^a \theta_j^b) &= \begin{cases}
        1 & \text{if $a=b=0$;}\\
        0 &\text{otherwise;}
    \end{cases} \\
    h_j(x_j^a\theta_j^b) &= \begin{cases}
        x_j^{a-1}\theta_j^{b+1} & \text{if $a \ge 1$ and $b=0$};\\
        0 & \text{otherwise}.
    \end{cases}
    \end{align*}
    One easily verifies that $(p_j,h_j)$ are contracting data.
\end{proof}

Given the contracting data from Lemma \ref{lem:contdata}, a minimal model $\cB$ for $\cB^{dg}$ can be constructed as in \cite[Section 7.2]{Sheridan2015}. 
As shown there, the underlying $R$-module for $\cB$ is isomorphic to that of $\cA$. 

\begin{lem}
    Let $\mu^*$ denote the $A_\infty$ products on $\cB$, and $\mu^2_{ext}$ the exterior algebra product. 
    Then $\Xi(\mu^* - \mu^2_{ext}) = W$, where $\Xi$ is the HKR isomorphism from \Cref{lem:hkr}.\end{lem}
\begin{proof}
     The argument follows that of \cite[Proposition 7.1]{Sheridan2015} closely. 
     It helps to have a formula for the contracting homotopy $h$. 
     For a basis element $z^a\theta^b \partial^K$, we define $A(a) := \max \{i:a_i \neq 0\}$ and $B(b) = \max \{i:b_i \neq 0\}$. 
     If $A(a) = j$, then we have
     $$h(z^a\theta^b\partial^K) = \begin{cases}
         \frac{\theta_{j}}{z_j} z^a \theta^b \partial^K & \text{ if $B(b) \le A(a) - 1$}\\
         0 & \text{otherwise.}
     \end{cases}$$
     Plugging this into the formula for $\tilde \mu^*:=\mu^* - \mu^2_{ext}$, one sees that the only contributions to $\tilde \mu^k(\partial_{i_1},\ldots,\partial_{i_k})$ come from the trees illustrated in Figure 14 of \cite{Sheridan2015} (this is proved by considering the filtration by $\theta$-degree, $|b|$). 
     Recall that the construction of $\mathcal{O}_0$ depends on a choice of $w_j \in S$ such that $\sum_j z_j w_j = W$. 
     For each monomial $z_1^{a_1}\ldots z_{|I|}^{a_{|I|}}$ appearing in $w_j$, we have a contribution to $\tilde \mu^*(\partial_{i_1},\ldots,\partial_{i_k})$ which vanishes unless $(i_1,\ldots,i_k) = (1,\ldots,1,2,\ldots,|I|,j)$ where there are $a_1$ copies of $1$, followed by $a_2$ copies of $2$, etc., up to $a_{|I|}$ copies of $I$, followed by $j$.  
    Plugging into the formula for $\Xi$, one sees that each such monomial yields a contribution of $z_jz^a$ to $\Xi(\tilde \mu^*)$, which gives the result.
\end{proof}
\subsection{Order zero}

Let $\cB_0 = \cB \otimes_{R_B} R_B/\fm_B$. 
This is a minimal model for the endomorphism algebra of $\mathcal{O}_0$ in the category of matrix factorizations of $W_0 = - z^{\vec{e}_{I}}$. 
We have an isomorphism $H(\cB_0) \cong H(\cA_0)$ on the level of cohomology, which upgrades to a quasi-isomorphism $\cB_0 \cong \cA_0$ by Corollary \ref{cor:hhdef} and \cite[Lemma 3.2]{Seidel2003}. 

\subsection{Applying versality}

We have proved the existence of a quasi-isomorphism $F_0:\cB_0 \to \cA_0$. We have also proved that the first-order deformation class corresponding to $\vec{p} \in P$, for both $\cA$ and $\cB$, is $\pm z^{\vec{p}}$ (in the case of $\cB$, this follows from the fact that $\Xi(\tilde{\mu}^*) = W$). 
In particular, we have $d_i \vec{e}_i \in P$ for all $i$, and the corresponding first-order deformation class is $\pm z_i^{d_i}$. 

\begin{prop}\label{prop:applyvers}
    There exists a mirror map $\Psi^*:R_B \to R_A$, sending $\nov_{\vec{p}} \mapsto \nov_{\vec{p}} \cdot \psi_{\vec{p}}$ where $\psi_{\vec{p}} = \pm 1 \text{ mod }\fm_A$ for all $\vec{p} \in P$, and a curved filtered quasi-isomorphism $F: \Psi^* \cB \to \cA$ with $F = F_0$ modulo $\fm_A$.
\end{prop}
\begin{proof}
    Follows from Proposition \ref{prop:vers}, whose hypotheses are verified by Lemmas \ref{lem:hh2} and \ref{lem:obst}. (It is also clear that $\cA$, $\cB$ are both free modules of finite rank, and in particular topologically free.)
\end{proof}

\subsection{Proof of Theorem \ref{thm:main}}

Let $\mathbf{q}:\G \to \Z$ be the morphism projecting $\Z \oplus M$ onto the $\Z$ factor. 
Let $\mathbb{B}$ denote the subcategory of the category of matrix factorizations consisting of all shifts of $\mathcal{O}_0$. 
Define $\tilde{\mathbb{A}}' = \mathbf{q}_* \hat{\mathbb{A}} = \mathbf{q}_* \mathbf{p}^* \mathbb{A}$, $\tilde{\mathbb{B}} = \mathbf{q}_* \mathbf{p}^* \mathbb{B}$. 
The curved filtered quasi-isomorphism from Proposition \ref{prop:applyvers} determines a curved filtered quasi-isomorphism $\Psi^*\tilde{\mathbb{B}} \to \tilde{\mathbb{A}}'$. 
As $\tilde{\mathbb{B}}$ is uncurved, it embeds canonically into $\tilde{\mathbb{B}}^{bc}$, by taking the zero bounding cochain on each object. 
Then we obtain, by \cite[Lemma 2.6]{Sheridan2017}, a filtered quasi-equivalence of uncurved $A_\infty$ categories, $\Psi^* \tilde{\mathbb{B}} \to (\tilde{\mathbb{A}}')^{bc}$, which is a quasi-isomorphism onto its image by a spectral sequence comparison argument. We denote the image by $\tilde{\mathbb{A}}$.

We have an embedding $\tilde{\mathbb{A}} \hookrightarrow \fuk(X,D)^{bc}$ by construction, and $\tilde{\mathbb{B}} \hookrightarrow \grmf_\Gamma(S,W)$ as in \cite[Section 4.4]{sheridan2021homological}. 
Passing to the fibre over the $\Lambda_{\Bbbk,Q}$-point $d(\lambda)$, and setting $b=\Psi(d(\lambda))$, we obtain a diagram
\begin{equation}
    \label{eq:corehms}
    \begin{tikzcd}
    \fuk(X,\omega;\Lambda_{\Bbbk,Q})^\bc & \grmf_\Gamma(\Lambda_{\Bbbk,Q}[z_i]_{i \in I},W_b)\\
\tilde{\mathbb{A}}_{d(\lambda)} \ar[u,hookrightarrow] \ar[r,leftrightarrow,"\sim"] & \tilde{\mathbb{B}}_b \ar[u,hookrightarrow]
\end{tikzcd}
\end{equation}

\begin{lem}\label{lem:sing_iso}
    If $\mathrm{char}(\Bbbk) \nmid \lcm(\lambda)$, then $W_b$ has isolated singular locus in the sense of \cite[Section 3.1]{Dyckerhoff2009}.
\end{lem}
\begin{proof}
    This follows by the proof of \cite[Proposition 4.4]{sheridan2021homological} using Proposition \ref{prop:tropreg} (which applies by our assumption on the characteristic of $\Bbbk$) in place of \cite[Theorem 4.5.1]{Maclagan2007}.
    \end{proof}

By Lemma \ref{lem:sing_iso}, we can apply the argument of \cite[Proposition 4.7]{sheridan2021homological} to show that $\tilde{\mathbb{B}}_b$ split-generates $\grmf_\Gamma(\Lambda_{\Bbbk,Q}[z_i]_{i \in I},W_b)$. 
Furthermore, as $W_b$ has an isolated singularity by Lemma \ref{lem:sing_iso}, \cite{Orlov2009} shows that this category of equivariant graded matrix factorizations is equivalent to $D^bCoh(\check{X}_b)$.

Note that $\check{X}_b$ is smooth and also tame as a stack because the stabilizers are finite closed subgroups of an algebraic torus. It therefore follows from \cite[Theorem 6.4]{Bergh2016} that $D^bCoh(\check{X}_b)$ is homologically smooth as a dg-category. Given this, automatic split-generation \cite{Ganatra2016, Sanda2021} (compare the proof of \cite[Proposition 4.8]{sheridan2021homological}) shows that $\tilde{\mathbb{A}}_{d(\lambda)}$ split-generates $\fuk(X,\omega;\Lambda_{\Bbbk,Q})^\bc$. This completes the proof of Theorem \ref{thm:main}.

\subsection{Proof of Theorem \ref{thm:int_mm}}

We will prove Theorem \ref{thm:int_mm} by verifying the hypotheses of \cite[Theorem C.20]{Sheridan2017}. 
We start by choosing $\lambda$ satisfying the MPCS condition; and the amb-nice cone $\Nef$ required for the definition of the ambient relative Fukaya category $\fuk(X,D)$. 
Let $R=\C[[\nov_{\vec{p}}]]_{\vec{p} \in P}$, and $\check{X}_R \subset \check{Y}_R$ be the $R$-scheme with defining equation
$$z^{\iota(1,0)} = \sum_{\vec{p} \in P} \nov_{\vec{p}} \cdot z^{\iota(1,\vec{p})}.$$
We must show that, for any `Novikov disc map' $d^*:R_A \to \Lambda_{\C,\R}$, if we set $b = \Psi(d)$, then the fibre $\check{X}_{\Psi(d)}$ is smooth and split-generated by $\tilde \cB_b$. 
Note that $\val(\Psi(d)) = \val(d) \in (\R_{>0})^P$ lies in the interior of $\Nef$, and hence satisfies the MPCS condition. 
This is sufficient for the proofs of smoothness of $\check{X}_b$, and split-generation of its derived category by $\tilde{\mathbb{B}}_b$, from the previous section to go through.

Applying \cite[Theorem C.20]{Sheridan2017}, we have that $\Psi$ is a Hodge-theoretic mirror map in the sense of \cite[Definition C.8]{Sheridan2017}. 
This uniquely characterizes the restriction of $\Psi^*$ to $R_{cl}$, the subalgebra in degree $0 \in H_1(X \setminus D) = M$, by the fact that it preserves flat coordinates, see \cite[Theorem C.11]{Sheridan2017}. 
Note that we can identify $R_{cl}$ with $\C[[K_{\ge 0}]]$ from the introduction. 
By \cite[Section 6.3.4]{coxkatz}, the unique mirror map which preserves flat coordinates is given by $\Phi|_{\C[[K_{\ge 0}]]}$, where $\Phi$ is the map from \eqref{eq:mm_formula} (see Appendix \ref{sec:comp_mm}). 
Thus we have $\Phi|_{\C[[K_{\ge 0}]]} = \Psi^*|_{\C[[K_{\ge 0}]]}$. 
This completes the proof of Theorem \ref{thm:int_mm}, because $\Psi^*$ maps $R_B$ to $R_A := \Z[[\NE_A]]$ by construction.

\appendix

\section{Versality result}\label{sec:vers}

In this section, we prove a versality result similar to \cite[Lemma 2.15]{Sheridan2017}. In contrast to that result, we remove an assumption that the coefficient ring is a field of characteristic zero, and take into account the obstruction map; on the other hand, we do not prove an equivariant version, cf. Remark \ref{rmk:equiv}.  

\subsection{Statement of the result}

We describe the setting for our deformation theory, which is a modification of the setup of \cite[Section 2]{Sheridan2017}.

Let $\Bbbk$ be a regular commutative ring. 
Let $\Z \to Y \to \Z/2$ be homomorphisms of abelian groups, whose composition is non-zero. 
All of our objects will be $Y$-graded; an object will be said to have degree $k \in \Z$ if its degree is the image of $k$ under $\Z \to Y$; and objects acquire $\Z/2$-gradings via the homomorphism $Y \to \Z/2$, which allows us to define Koszul signs.

Let $P$ be a finite set, and $P \to Y: p \mapsto y_p$ be an injective map with all $y_p$ lying in the preimage of $0 \in \Z/2$. The map induces a homomorphism $y:\Z^P \to Y$ which sends $e_p$ (the $p$th generator of $\Z^P$) to $y_p \in Y$.

Let $\NE_\R \subset \R^P$ be a strongly convex cone with vertex at the origin, and $\NE_A = \NE_\R \cap \Z^P$. 
We assume that $\NE_A$ is \emph{nice}, in the following sense:

\begin{defn}\label{def:nice}
We say that $\NE_A$ is \emph{nice} if for all $p \in P$:
\begin{enumerate}
    \item[] \textbf{(Nice $1_p$)} \label{it:nice_1} $e_p$ lies in $\NE_A$;
    \item[] \textbf{(Nice $2_p$)} \label{it:nice_2} if $e_p = u+v$ with $u$ and $v$ in $\NE_A$, then one of $u$ or $v$ is $0$;
    \item[] \textbf{(Nice $3_p$)} \label{it:nice_3} if $u \in \NE_A$ and $y(u) = y_p$, then $u - e_p \in \NE_A$.
\end{enumerate}
Note that this condition implicitly depends on the choice of map $y$.
\end{defn}

We define $\NE_B = (\Z_{\ge 0})^P$. While $\NE_B$ satisfies \textbf{(Nice $1_p$)} and \textbf{(Nice $2_p$)} for all $p$, it need not satisfy \textbf{(Nice $3_p$)}.

We consider the algebras $\tilde{R}_B = \Bbbk[\NE_B]$, respectively $\tilde{R}_A = \Bbbk[\NE_A]$, whose elements are finite $\Bbbk$-linear combinations of monomials $\nov^u$ for $u \in \NE_B$, respectively $\NE_A$. 
We define $\nov_p := \nov^{e_p}$. 
We equip $\tilde{R}_B$ and $\tilde{R}_A$ with $Y$-gradings by putting $\nov^u$ in degree $-y(u)$. 
Let $\tilde{\fm}_B$ be the ideal corresponding to the origin of the cone, and let $R_B = \Bbbk[[\NE_B]]$ be the $Y$-graded $\tilde\fm_B$-adic completion of $\tilde{R}_B$; define $R_A = \Bbbk[[\NE_A]]$ similarly. 

\begin{rmk}
    Our assumption that $\NE_A$ is nice is equivalent to the fact that $R_A$ is nice in the sense of \cite[Definition 2.3]{Sheridan2017}: \textbf{(Nice $1_p$)} says that $\nov_p$ lies in $\tilde \fm_A$; \textbf{(Nice $2_p$)} says that $\nov_p \notin \tilde\fm_A^2$; and \textbf{(Nice $3_p$)} says that $\nov_p$ generates the $-y_p$-graded piece of $\tilde R_A$ as a $(\tilde R_A)_0$-module, where $(\tilde R_A)_0$ denotes the graded piece of $\tilde R_A$ in degree $0 \in Y$.
\end{rmk}

\begin{defn}\label{def:top_free_mod}
    Let $R$ be a commutative $Y$-graded $\Bbbk$-algebra, $\fm \subset R$ an ideal such that the $\fm$-adic filtration is graded complete and $R/\fm \simeq \Bbbk$, and $M$ a graded $R$-module. 
    We say that $M$ is \emph{topologically free} if
    \begin{itemize}
        \item $M /\fm M$ is a free $\Bbbk$-module;
        \item there exists an isomorphism $M \simeq (M/\fm M) \hat{\otimes} R$ lifting the natural map $M \to M/\fm M$, where $(M/\fm M) \hat{\otimes} R$ denotes the graded completion of $(M/\fm M) \otimes R$ with respect to the $\fm$-adic filtration. 
    \end{itemize}
\end{defn}

\begin{defn}[Cf. Remark 2.2 in \cite{Seidel2003}]\label{def:top_cat}
    Let $(R,\fm)$ be as in \Cref{def:top_free_mod}. A \emph{topological} $R$-linear $A_\infty$ category is a curved $A_\infty$ category whose hom-spaces are topologically free $R$-modules and composition maps are $R$-multilinear, with curvature $\mu^0_X \in \fm \cdot \hom(X,X)$ for any object $X$.
\end{defn}

Note that the $\fm$-adic filtration equips any topological $R$-linear $A_\infty$ category with the structure of a curved filtered $R$-linear $A_\infty$ category, in the sense of \cite[Definition 2.1]{perutz2022constructing}. 
Thus we may define notions of bounding cochain, curved filtered $A_\infty$ functor, and curved filtered quasi-equivalence for topological $A_\infty$ categories, as in \cite[Section 2]{perutz2022constructing}.

Associated to a topological $R$-linear $A_\infty$ category $\cA$, there is an uncurved graded $\Bbbk$-linear $A_\infty$ category $\cA_0 := \cA/\fm \cA$, whose morphism spaces are free $\Bbbk$-modules.  
Associated to a curved filtered $A_\infty$ functor $F: \cA \to \cB$ between topological $A_\infty$ categories, there is an uncurved $\Bbbk$-linear $A_\infty$ functor $F_0:\cA_0 \to \cB_0$; and $F$ is a curved filtered quasi-equivalence if and only if $F_0$ is a quasi-equivalence.

\begin{rmk}\label{rmk:finite_top_free}
    A free $R$-module of finite rank is topologically free. 
    In the present work, all topologically free $R$-modules we consider will in fact have finite rank; nevertheless, we expect the generalization to arbitrary rank to be helpful for future applications.
\end{rmk}

Let $\cB$ be a $Y$-graded topological $R_B$-linear $A_\infty$ category. 
Similarly, let $\cA$ be a $Y$-graded topological $R_A$-linear $A_\infty$ category.
We suppose that there exists an $A_\infty$ quasi-equivalence $F_0:\cB_0 \to \cA_0$.

Because $\NE_B$ satisfies \textbf{(Nice $1_p$)} and \textbf{(Nice $2_p$)}, we have a well-defined first-order deformation class $b_p \in \HH^{2+y_p}(\cB_0)$ of $\cB$ for all $p \in P$; and similarly $a_p \in \HH^{2+y_p}(\cA_0)$. 
For any $p \in P$, we define the $p$th `obstruction map'
\begin{align*}
    Obs_p: \HH^{2}(\cA_0) & \to \HH^{3+y_p}(\cA_0)\\
    Obs_p(\alpha) &:= [a_p,\alpha],
\end{align*}
where $[-,-]$ denotes the Gerstenhaber bracket; and the `total obstruction map'
\begin{align*}
    Obs: \HH^2(\cA_0) &\to \bigoplus_{p \in P} \HH^{3+y_p}(\cA_0),\\
    Obs(\alpha) &:= \bigoplus_{p \in P}Obs_p(\alpha).
\end{align*}
We can now state our versality result.

\begin{prop}\label{prop:vers}
In the above situation, suppose that:
\begin{enumerate}
   \item $\HH^{2+y_p}(\cA_0)$ is a free $\Bbbk$-module of rank $1$, spanned by $a_p$, for all $p \in P$, and similarly for $\cB$;
    \item $Obs$ is injective;
    \item $\HH^{2+y}(\cA_0) = 0$ for $y \in y(\NE_A) \setminus \left( \{0\} \cup \{y_p\}_{p \in P}\right)$.
\end{enumerate}
Then there exists a $Y$-graded $\Bbbk$-algebra homomorphism $\Psi^*: R_B \to R_A$, sending $\nov_p \mapsto \nov_p \cdot \psi_p$ for some units $\psi_p \in R_A$ of degree $0 \in Y$, together with a curved filtered quasi-equivalence $F: \Psi^*\cB \to \cA$, with $F = F_0$ modulo $\fm$. 
\end{prop}

\begin{rmk}\label{rmk:equiv}
    One could generalize Proposition \ref{prop:vers} by working equivariantly with respect to a (signed) group action, as was done in \cite{Sheridan2017}. However this is unnecessary for the application in this paper, and complicates all statements and proofs, so we have not done it.
\end{rmk}

The proof of \Cref{prop:vers} occupies the remainder of this appendix.

\subsection{Pre-functors}

We will construct the functor $F$ from Proposition \ref{prop:vers} by starting with a map $F$ which is not a functor, then iteratively `correcting' it so that it satisfies the $A_\infty$ functor equation to successively higher orders in the $\fm$-adic filtration. 
Thus it is necessary to develop a little bit of theory for maps which are not $A_\infty$ functors, which we call `pre-functors'.

A graded filtered $R$-linear pre-$A_\infty$ category\footnote{Note that a pre-$A_\infty$ category is not the same thing as an $A_\infty$-pre-category in the sense, e.g., of \cite{Kontsevich-Soibelman-syz}.} $\cC$ consists of a set of objects, together with a graded $R$-module $\hom(X,Y)$ for each pair of objects $X$, $Y$. We assume that the $\fm$-adic filtration on each $\hom$-space is complete (we don't need the assumption of topological freeness yet).

Given two such pre-$A_\infty$ categories $\cC$, $\cD$, and maps $F_0,F_1:\ob \cC \to \ob \cD$, we define
\begin{align*}CC^\bullet&(\cC,(F_0 \otimes F_1)^*\cD) \\&:= \prod_{X_0,\ldots,X_s} \Hom\left(\hom(X_0,X_1)[1] \otimes \ldots \otimes \hom(X_{s-1},X_s)[1],\hom(F_0X_0,F_1X_s)[1]\right)[-1].
\end{align*}
If $F_0=F_1=F$, then we replace `$(F_0 \otimes F_1)^*$' with `$F^*$' in the notation; and if $F=\id$, we write $CC^\bullet(\cC) := CC^\bullet(\cC,\id^*\cC)$.

We define a \emph{pre-functor} from $\cC$ to $\cD$ to be a map on objects $F: \ob \cC \to \ob \cD$, together with an element $F \in CC^1(\cC,F^*\cD)$, whose length-zero component $F^0$ lies in $\fm \cdot \cD(FX,FX)$ for all $X$.
We define a pre-$A_\infty$ category $\mfun(\cC,\cD)$, whose objects are pre-functors from $\cC$ to $\cD$, with 
$$\hom_{\mfun(\cC,\cD)}(F_0,F_1) := CC^\bullet(\cC,(F_0 \otimes F_1)^*\cD).$$
Note that a pre-functor $F$ can also be considered as an element of $\hom_{\mfun(\mathcal C, \mathcal D)}(F,F)$; we will sometimes do so implicitly, trusting that it will be clear from the context when we are doing so.

Given pre-$A_\infty$ categories $\cC_1,\cC_2,\cC_3$, maps $G_0,G_1: \ob \cC_2 \to \ob \cC_3$ together with $\psi \in CC^\bullet(\cC_2,(G_0\otimes G_1)^*\cC_3)$, and pre-functors $F_0,F_1,\ldots,F_k \in \mfun(\cC_1,\cC_2)$ together with $\phi_i \in \hom_{\mfun}(F_{i-1},F_i)$ for $i=1,\ldots,k$, we define 
$$\psi\{\phi_1,\ldots,\phi_k\} \in CC^\bullet\left(\cC_1,(G_0 \circ F_0 \otimes G_1 \circ F_k)^*\cC_3\right)$$
by
\begin{multline}\label{eq:brace}\psi\{\phi_1,\ldots,\phi_k\}(c_1,\ldots,c_s) := 
\sum (-1)^\dagger \psi\left(F_0^*(c_1,\ldots,c_{s_0^1}),\ldots,F_0^*(\ldots,c_{s_0^{j_0}}),\phi_1^*(\ldots,c_{t_1}),\right.\\
\left.F_1^*(\ldots),\ldots \ldots,F_{k-1}^*(\ldots,c_{s_{k-1}^{j_{k-1}}}),\phi_k^*(\ldots,c_{t_k}),F_k^*(\ldots),\ldots,F_k^*(\ldots,c_{s_k^{j_k}})\right)
\end{multline}
where the sum is over all $j_0,\ldots,j_k \ge 0$ and all 
$$ s_0^1 \le s_0^2 \le \ldots \le s_0^{j_0} \le t_1 \le s_1^1 \le \ldots \le s_1^{j_1} \le t_2 \le \ldots \le t_k \le s_k^1 \le \ldots \le s_k^{j_k} = s;$$
and the sign is 
$$\dagger = \sum_{i=1}^k (|\phi_i|-1) \cdot \left(\sum_{\ell=1}^{s_{i-1}^{j_{i-1}}} |c_\ell|'\right).$$
Here $|c_i|'$ denotes the degree of $c_i$ in the shifted complex $\hom(X_{i-1},X_i)[1]$.
We have
$$\left|\psi\{\phi_1,\ldots,\phi_k\}\right| = |\psi| + \sum_{i=1}^k |\phi_i| - k.$$
Note that the sum \eqref{eq:brace} is potentially infinite, but our assumptions that the length-zero components $F^0_i$ are of order $\fm$, and the filtrations are complete, ensure that it converges. 
We have chosen to omit the $F_i$ and $G_i$ from the notation to avoid clutter, as it will be clear from the context what they are. 
However, we make an exception in the case $k=0$ when there is only one pre-functor $F_0$: in this case we will write $\psi\{\}_{F_0}$. 

If $F:\cC_1 \to \cC_2$ and $G: \cC_2 \to \cC_3$ are pre-functors, then we define the pre-functor $G \circ F: \cC_1 \to \cC_3$ to be given by composition on the level of objects, together with the class $G\{\}_F \in CC^1(\cC_1,(G \circ F)^* \cC_3)$. In the case that $G_0$ and $G_1$ are endowed with the structure of pre-functors, we regard $\psi\{\phi_1,\ldots,\phi_k\}$ as an element of $\hom_{\mfun}(G_0 \circ F_0,G_1 \circ F_k)$.

In the case that all $G_i$ and $F_i$ are the identity, we recover the usual brace operations on $CC^\bullet(\cC)$ (see, e.g., \cite{Getzler_GM}). 

Observe that a (filtered, curved, graded) $A_\infty$ structure on the pre-$A_\infty$ category $\cC$ is a class $\mu \in \hom^2_{\mfun(\cC,\cC)}(\id,\id)$ satisfying $\mu\{\mu\} = 0$, and whose length-zero component $\mu^0_X$ lives in $\fm \cdot \hom(X,X)$, for all $X$. 

Given two $A_\infty$ categories $\cC$ and $\cD$, a (non-unital, filtered, curved) functor $F:\cC \to \cD$ is a pre-functor which satisfies the $A_\infty$ functor equation $\delta(F) = 0$, where
$$\delta(F) := \mu_\cD\{\}_F - F\{\mu_\cC\}.$$
Let $\nufun(\cC,\cD) \subset \mfun(\cC,\cD)$ be the full sub-pre-$A_\infty$ category whose objects are the functors. 
We define an (uncurved) $A_\infty$ structure on $\nufun(\cC,\cD)$ by
$$\mu^k_{\nufun(\cC,\cD)}(\alpha_1,\ldots,\alpha_k) := \left\{ \begin{array}{ll}
    0 & k=0,\\
    \mu_\cD\{\alpha_1\} + (-1)^{|\alpha_1|} \alpha_1 \{\mu_\cC\} & k=1,\\
    \mu_\cD\{\alpha_1,\ldots,\alpha_k\} & k>1.
    \end{array}\right.$$
We define the category $\NuFun(\cC,\cD) := H(\nufun(\cC,\cD))$.
				
Note that $h \in \hom^*_{\nufun}(F_0,F_1)$ is closed, and hence defines a morphism $[h] \in \Hom^*_{\NuFun}(F_0,F_1)$, if and only if $\eps(h) = 0$, where
$$\eps(h) := \mu_\cD\{h\} + (-1)^{|h|} h\{\mu_\cC\}.$$
Even if $F_0$ and $F_1$ are merely pre-functors, we may still define $\eps(h)$ by the same formula.

The brace operations satisfy formulae analogous to those satisfied by the analogous operations on the Hochschild complex of a single $A_\infty$ category (see, e.g., \cite{Getzler_GM}), among which we need the following:

\begin{lem}
If $F:\cC_1 \to \cC_2$ is a pre-functor, and $\psi_i,\phi_i \in \hom_{\nufun(\cC_i,\cC_i)}(\id,\id)$ for $i=1,2$, then:
\begin{align}
\label{eq:brace1}\psi_2\{\phi_2\}\{\}_F &= \psi_2 \{\phi_2\{\}_F\};\\
\label{eq:brace2}\psi_2\{\}_F\{\phi_1\} &= \psi_2\{F\{\phi_1\}\};\\
\label{eq:brace3}F\{\psi_1\}\{\phi_1\} &= (-1)^{(|\psi|-1)\cdot(|\phi|-1)} F\{\phi_1,\psi_1\} + F\{\psi_1\{\phi_1\}\}  + F\{\psi_1,\phi_1\}.
\end{align}
If $G:\cC_1 \to \cC_2$ is another pre-functor, and $h\in \hom_{\mfun}(F,G)$, then furthermore:
\begin{align}
\label{eq:brace4}\psi_2\{\phi_2\}\{h\} &= (-1)^{(|h|-1)\cdot(|\phi_2|-1)}\psi_2\{h,\phi_2\{\}_G\} + \psi_2\{\phi_2\{h\}\} + \psi_2\{\phi_2\{\}_F,h\};\\
\label{eq:brace5}\psi_2\{h\}\{\phi_1\} &= (-1)^{(|h|-1)\cdot(|\phi_1|-1)} \psi_2\{F\{\phi_1\},h\} + \psi_2\{h\{\phi_1\}\} +  \psi_2\{h,G\{\phi_1\}\};\\
\label{eq:brace6}h\{\psi_1\}\{\phi_1\} &= (-1)^{(|\psi_1|-1)\cdot(|\phi_1|-1)}h\{\phi_1,\psi_1\} + h\{\psi_1\{\phi_1\}\} + h\{\psi_1,\phi_1\}.
\end{align}

\end{lem} 
\qed

\begin{lem}\label{lem:epsdel}
If $F:\cC_1 \to \cC_2$ is a pre-functor between $A_\infty$ categories, then $\eps(\delta(F)) = 0$.
\end{lem}
\begin{proof}
Applying \eqref{eq:brace1}--\eqref{eq:brace3}, and observing that $|\delta(F)|=|\mu_i|=2$, we have:
\begin{align*}
\eps(\delta(F)) &= \mu_2\{\mu_1\{\}_F\} - \mu_2\{F\{\mu_1\}\} + \mu_2\{\}_F\{\mu_1\} - F\{\mu_1\}\{\mu_1\} \\
&= \mu_2\{\mu_1\}\{\}_F - \mu_2\{\}_F\{\mu_1\} + \mu_2\{\}_F\{\mu_1\} - F\{\mu_1,\mu_1\} + F\{\mu_1\{\mu_1\}\} + F\{\mu_1,\mu_1\} \\
& = 0,
\end{align*}
where in the last step we also used the $A_\infty$ relations $\mu_i\{\mu_i\} = 0$.
\end{proof}

\begin{lem}\label{lem:epseps}
If $F,G:\cC_1 \to \cC_2$ are pre-functors between $A_\infty$ categories, and $h\in \hom_{\mfun}(F,G)$, then 
$$\eps(\eps(h)) = (-1)^{|h|}\mu_2\{h,\delta(G)\} - \mu_2\{\delta(F),h\}.$$
\end{lem}
\begin{proof}
We have
$$ \eps(\eps(h)) = \mu_2\{\mu_2\{h\}\} + (-1)^{|h|} \mu_2\{h\{\mu_1\}\} + (-1)^{|h|-1} \mu_2\{h\}\{\mu_1\} - h\{\mu_1\}\{\mu_1\}.$$
Applying \eqref{eq:brace4}--\eqref{eq:brace6}, and using that $|\mu_i| = 2$, we find that this is equal to the expression
\begin{multline}
\left(\mu_2\{\mu_2\}\{h\} + (-1)^{|h|}\mu_2\{h,\mu_2\{\}_G\} - \mu_2\{\mu_2\{\}_F,h\}\right) \\
+ (-1)^{|h|}\left(\mu_2\{h\}\{\mu_1\} + (-1)^{|h|}\mu_2\{F\{\mu_1\},h\} - \mu_2\{h,G\{\mu_1\}\}\right)\\
+ (-1)^{|h|-1}\left(\mu_2\{h\}\{\mu_1\}\right) + \left(h\{\mu_1,\mu_1\} - h\{\mu_1\{\mu_1\}\} - h\{\mu_1,\mu_1\}\right).
\end{multline}
Grouping and cancelling terms (again using the $A_\infty$ relations $\mu_i\{\mu_i\} = 0$) gives the result.
\end{proof}

\subsection{Composition}

We review the composition of (pre-)functors, following \cite[Section 1e]{seidel2008fukaya}.

Given pre-$A_\infty$ categories $\cC_1,\cC_2,\cC_3$ and a pre-functor $G:\cC_2 \to \cC_3$, we obtain a `left-composition' pre-functor 
$$\cL_G: \mfun(\cC_1,\cC_2) \to \mfun(\cC_1,\cC_3)$$
which on the level of objects sends $F \mapsto G\circ F$, and on the level of morphisms is given by
$$\cL^i_G(h_1,\ldots,h_k) := G\{h_1,\ldots,h_k\}.$$
If the $\cC_i$ are $A_\infty$ categories, and $G$ is an $A_\infty$ functor, then $\cL_G$ defines an $A_\infty$ functor $\nufun(\cC_1,\cC_2) \to \nufun(\cC_1,\cC_3)$.

On the other hand, given a pre-functor $F:\cC_1 \to \cC_2$, we obtain a `right-composition' pre-functor
$$\cR_F:\mfun(\cC_2,\cC_3) \to \mfun(\cC_1,\cC_3)$$
which on the level of objects sends $G \mapsto G \circ F$, and on the level of morphisms is given by
$$\cR^1_F(h) = h\{\}_F,$$
with $\cR^i_F = 0$ for $i \neq 1$. 
The following is straightforward from the definitions: 

\begin{lem}\label{lem:R1eps}
    We have
    $$\cR^1_F(\eps(h)) = \eps(\cR^1_F(h)) + (-1)^{|h|} h\{\delta(F)\}$$
whenever the expression makes sense.
\end{lem}

In particular, if the $\cC_i$ are $A_\infty$ categories, and $F$ is an $A_\infty$ functor, then $\cR^1_F$ is a chain map. 
 In fact, it defines an $A_\infty$ functor.

Note that $\cL$ and $\cR$ can be extended to a bifunctor $\nufun(\cC_1,\cC_2) \times \nufun(\cC_2,\cC_3) \to \nufun(\cC_1,\cC_3)$ \cite{Lyubashenko_multi}, but we will not use this.

\subsection{Proof of Proposition \ref{prop:vers}}

The setting for the proof of Proposition \ref{prop:vers} is the cochain complex $C_0 := \hom_{\nufun(\cB_0,\cA_0)}(F_0,F_0)$. 

\begin{lem}\label{lem:LF_RF}
    The natural maps
    \begin{equation}\label{eq:LF_RF}
 CC(\cB_0) = \hom_{\nufun(\cB_0,\cB_0)}(\id,\id) \xrightarrow{\cL^1_{F_0}} C_0 \xleftarrow{\cR^1_{F_0}} \hom_{\nufun(\cA_0,\cA_0)}(\id,\id) = CC(\cA_0)
 \end{equation}
 are quasi-isomorphisms.
\end{lem}
\begin{proof} 
    We follow the argument from \cite[Lemma 1.7]{seidel2008fukaya} using the spectral sequences arising from the length filtration on Hochschild cohomology. For any collection of objections $X_0,X_1,\cdots,X_s$ in $\ob\cB_0$, we use the shorthand $$ \cB_0(X_0,X_1,\cdots, X_s):=\operatorname{hom}_{\cB_{0}}(X_0,X_1)[1]\otimes \operatorname{hom}_{\cB_{0}}(X_1,X_2)[1]\otimes \cdots \otimes  \operatorname{hom}_{\cB_{0}}(X_{s-1},X_s)[1] $$
 where the differential on the right-hand side is given by the natural tensor product of complexes. We define $\cA_0(F_0X_0,\ldots,F_0X_s)$ similarly.    
    We begin by proving that $\cL^1_{F_0}$ is a quasi-isomorphism. The length filtration on $C_0$ and $CC(\cB_0)$ gives rise to two spectral sequences $E_r^{pq}(C_0)$ and $E_r^{pq}(CC(\cB_0)).$ The map $\cL^1_{F_0}$ induces a map between these spectral sequences which on the first page is given by the direct product of natural maps: 
\begin{align} \label{eq:mapE1} H^*(\operatorname{Hom}_\Bbbk(\cB_0(X_0,X_1,\cdots, X_s), \cB_0(X_0,X_s))) \to  H^*(\operatorname{Hom}_\Bbbk(\cB_0(X_0,X_1,\cdots, X_s), \cA_{0}(F_0X_0,F_0X_s))) \end{align}  

Because $\Bbbk$ is a regular commutative ring, any complex of projective modules over $\Bbbk$ is K-projective in the sense of \cite{Spaltenstein88} (see \cite[Proposition 4.1(b)]{Positselski2021}). In particular, $\cB_0(X_0,X_1,\cdots, X_s)$ is K-projective, and hence the map \eqref{eq:mapE1} is an isomorphism (apply the definition of K-projectivity to the cone of $F_0:\cB_0(X_0,X_s) \to \cA_0(F_0X_0,F_0X_s)$). From this, we deduce that $\cL^1_{F_0}$ is a quasi-isomorphism as claimed. 

To see that $\cR^1_{F_0}$ is a quasi-isomorphism, note that the maps $$ \cB_0(X_0,X_1,\cdots, X_s) \to \cA_0(F_0X_0,F_0X_1,\cdots, F_0X_s)$$ remain quasi-isomorphisms because K-projective complexes are K-flat by \cite[Proposition 5.8]{Spaltenstein88}. The mapping cone is therefore an acyclic K-projective complex which is therefore contractible (see \cite[Section 1.1]{Spaltenstein88}). It follows that the map induced by $\cR^1_{F_0}$ on the first pages of the length spectral sequence is also an isomorphism, which proves the claim. 
\end{proof}

We choose a graded isomorphism of pre-categories, $\cB \simeq \cB_0 \hat{\otimes} R_B$, and similarly for $\cA$. 
These exist by our assumption that the morphism spaces are topologically free. 
We henceforth implicitly identify $\cB$ with $\cB_0 \hat{\otimes} R_B$, and $\cA$ with $\cA_0 \hat{\otimes} R_A$, via these isomorphisms. 
This allows us to take the `Taylor expansion' of all of the maps that concern us. 
For example, we may expand the $A_\infty$ structure map $\mu_\cB$ as 
\begin{align*} 
\mu_\cB &= \sum_{u \in \NE_B} \mu_{\cB,u} \cdot \nov^u, \quad \text{where}\\
\mu_{\cB,u} & \in CC^{2+y(u)}(\cB_0).
\end{align*}
We may similarly expand $\mu_\cA$, with Taylor coefficients $\mu_{\cA,u} \in CC^{2+y(u)}(\cA_0)$ for $u \in \NE_A$.

In order to prove Proposition \ref{prop:vers} we will construct:
\begin{itemize}
\item units $\psi_p \in R_A$ of degree $0$, from which we define $\Psi^*$ by setting $\Psi^*(\nov_p) = \nov_p \cdot \psi_p$;
\item $F \in \nufun(\Psi^*\cB,\cA)$ such that $F=F_0$ modulo $\fm$.
\end{itemize}

We will expand 
\begin{align*}
    \psi_p &= \sum_{u \in \NE_A,y(u)=0} \psi_{p,u} \cdot \nov^u, \qquad \text{where $\psi_{p,u} \in \Bbbk$} \\
    F & = \sum_{u \in \NE_A} F_u \cdot \nov^u, \qquad \text{ where $F_u \in C_0^{1+y(u)}$.}
\end{align*} 
Note that the infinite sum defining $F$ converges for any choice of such Taylor coefficients $F_u$, because the $\fm$-adic filtration on each morphism space of $\cA$ is complete by our assumption that it is topologically free. 

We will also expand 
\begin{align*}
    \delta(F) &= \sum_{u \in \NE_A} \delta(F)_u \cdot \nov^u, \qquad \text{where $\delta(F)_u \in C_0^{2+y(u)}$.}
\end{align*}  

Finally, we expand the map 
\begin{align*}
    \eps: hom^*_{\mfun(\Psi^*\cB,\cA)}(F,F) & \to hom^{*+1}_{\mfun(\cB,\Psi^*\cA)}(F,F)\qquad \text{as}\\
    \eps & = \sum_{u \in \NE_A} \eps_u \cdot \nov^u, \quad \text{where}\\
    \eps_u : C_0^* & \to C_0^{*+1+y(u)}.
\end{align*}

\begin{lem}\label{lem:eps_obs}
    Suppose that $\delta(F)_{e_p} = 0$. Then
    \begin{enumerate}
        \item \label{it:eps_p_chain} we have $\eps_{e_p} \eps_0 + \eps_0 \eps_{e_p} = 0$, so that there is a well-defined map
        $$[\eps_{e_p}]: H^*(C_0)  \to H^{*+1+y_p}(C_0);$$
        \item \label{it:eps_p_obs} we have
        $$[\eps_{e_p}] \circ [\cR^1_{F_0}] = [\cR^1_{F_0}] \circ Obs_p.$$
    \end{enumerate}
\end{lem}
\begin{proof}
    By Lemma \ref{lem:epseps}, we have 
    $$\eps(\eps(h)) = (-1)^{|h|}\mu_{\cA}\{h,\delta(F)\} + \mu_{\cA}\{\delta(F),h\}.$$
    Taking the $\nov_p$ Taylor coefficient of this equation, and using the hypotheses that $\delta(F)_0 = \delta(F_0) = 0$ and $\delta(F)_{e_p} = 0$ together with the assumption that $\NE_A$ satisfies \textbf{(Nice $2_p$)}, gives \eqref{it:eps_p_chain}.

    Now let us define
    \begin{align*}
        \tilde{\eps}: CC^*(\cA) & \to CC^{*+1}(\cA),\\
        \tilde{\eps}(h) &= \mu_\cA \{h\} +(-1)^{|h|} h \{\mu_\cA\}.
    \end{align*}
    Let $\tilde{\eps}_u: CC^*(\cA_0) \to CC^{*+1+y(u)}(\cA_0)$ denote the Taylor coefficients of $\tilde{\eps}$. 
    From the definitions, $\tilde{\eps}_0$ is equal to the Hochschild differential on $CC^*(\cA_0)$, while $\tilde{\eps}_{e_p}$ is equal to $Obs_p$.
    
    By Lemma \ref{lem:R1eps}, we have
    $$\cR^1_F(\tilde{\eps}(h)) = \eps(\cR^1_F(h)) + (-1)^{|h|}h\{\delta(F)\}.$$
    Taking the $\nov_p$ Taylor coefficient of this equation, and using the same assumptions as above, gives
    $$\cR^1_{F_0} \circ \tilde{\eps}_{e_p} + (\cR^1_F)_{e_p} \circ \tilde{\eps}_0 = \eps_{e_p} \circ \cR^1_{F_0} + \eps_0 \circ (\cR^1_F)_{e_p}.$$
    Thus, $\cR^1_{F_0} \circ Obs_p$ is homotopic to $\eps_{e_p} \circ \cR^1_{F_0}$ via the homotopy $(\cR^1_F)_{e_p}$, which yields \eqref{it:eps_p_obs}. 
\end{proof}

It does not quite work to construct $\psi_{p,u}$ and $F_u$ order-by-order with respect to the $\fm$-adic filtration. 
Rather, we need to work order-by-order with respect to a slightly different partial order on $\NE_A$:

\begin{lem}\label{lem:porder}
    There exists a partial order $\le$ on $\NE_A$ with the following properties:
    \begin{itemize}
        \item if $v-u \in \NE_A$, then $u \le v$;
        \item if $y(v) = 0$ and $v-u+e_p \in \NE_A\setminus \{0\}$, then $u \le v$;
        \item $(\NE_A)^{\leq v} =  \{u: u \le v\}$ is finite for all $v \in \NE_A$. That is, $\NE_A$ is cofinite.
    \end{itemize}
\end{lem}
\begin{proof}
    Let us define $u \le_1 v$ if $v-u \in \NE_A$, and $u \le_2 v$ if $y(v) = 0$ and $v-u+e_p \in \NE_A \setminus \{0\}$ for some $p \in P$. 
    We define $u\le v$ if and only if at least one of the following holds: $u \le_1 v$; or there exists $w \in \NE_A$ such that $u \le_2 w \le_1 v$.  
    Reflexivity ($u \le u$ for all $u$) is clear; it remains to check antisymmetry ($u \le v \le u$ implies $u=v$) and transitivity ($u \le v \le w$ implies $u \le w$). 
    
    In order to do this, we start by observing that:
    \begin{enumerate}
        \item \label{it:le1le1} if $u \le_1 v \le_1 w$ then $u \le_1 w$;
        \item \label{it:le1le2} if $u \le_1 v \le_2 w$ then $u \le_2 w$;
        \item \label{it:0le2} if $y(u) = 0$ and $u \le_2 v$ then $u \le_1 v$. 
    \end{enumerate}
    Only \eqref{it:0le2} requires an argument. Note that by definition, $v-u+e_p \in \NE_A \setminus \{0\}$. As $y(u) = y(v) = 0$, we have $y(v-u+e_p) = y_p$. As $\NE_A$ satisfies \textbf{(Nice $3_p$)}, this implies that $v-u \in \NE_A$, which implies $u \le_1 v$ as required.

    Now suppose that $u \le v \le u$. We split into cases:
    \begin{itemize}
    \item if $u \le_1 v \le_1 u$, then $u=v$ because $\NE_A$ is strongly convex.
    \item if $u \le_1 v \le_2 w \le_1 u$, then we have $e_p = (v-u) + (w-v+e_p) + (u-w)$ with each bracketed term lying in $\NE_A$, and the middle one being non-zero. As $\NE_A$ satisfies \textbf{(Nice $2_p$)} and is strongly convex, this implies that $v-u = w-u = 0$, in particular $u=v$. 
    \item if $u \le_2 w \le_1 v \le_1 u$, then $e_p = (w-u+e_p) + (v-w) + (u-v)$ and the same argument shows that $u-v = v-w = 0$, in particular $u=v$.
    \item if $u \le_2 w \le_1 v \le_2 w' \le_1 u$, then we have
    \begin{align*}
        &u \le_2 w \le_2 w' \le_1 u \qquad \text{by \eqref{it:le1le2}} \\
        \Rightarrow & u \le_2 w \le_1 w' \le_1 u \qquad \text{ by \eqref{it:0le2} (as $y(w) = 0$)}\\
        \Rightarrow &u=w \qquad \text{by the third case.}
    \end{align*}
    Thus $u \le_1 v \le_2 w' \le_1 u$, so $u=v$ by the second case.
    \end{itemize}
    In each case, we have proved $u=v$, so $\le$ is antisymmetric.

    Now suppose that $u \le v \le w$. We split into cases:
    \begin{itemize}
        \item if $u \le_1 v \le_1 w$, then $u \le_1 w$ by \eqref{it:le1le1}, so $u \le w$.
        \item if $u \le_1 v \le_2 w' \le_1 w$, then $u \le_2 w' \le_1 w$ by \eqref{it:le1le2}, so $u \le w$.
        \item if $u \le_2 w' \le_1 v \le_1 w$, then $u \le_2 w' \le_1 w$ by \eqref{it:le1le1}, so $u \le w$.
        \item if $u \le_2 v' \le_1 v \le_2 v'' \le_1 w$, then we have
        \begin{align*}
        &u \le_2 v' \le_2 v'' \le_1 w \qquad \text{by \eqref{it:le1le2}}\\
    \Rightarrow &u \le_2 v' \le_1 v'' \le_1 w \qquad \text{by \eqref{it:0le2} (as $y(v') = 0$)}\\
    \Rightarrow &u \le_2 v' \le_1 w \qquad \text{by \eqref{it:le1le1}}\\
    \Rightarrow & u \le w.
    \end{align*}
   \end{itemize}
   In each case, we have proved $u \le w$, so $\le$ is transitive. 
   This completes the proof that $\le$ is a partial order; it clearly has the first two desired properties.

     In order to establish the final property, let $\lambda$ be an element of the interior of the dual cone to $\NE_\R$ (which is non-empty as $\NE_A$ is strongly convex). We note that $u \le v$ implies an upper bound
     $$\lambda(u) \le \lambda(v) + \max_{p \in P} \lambda(e_p).$$
     The property now follows as the region $\NE_\R \cap \{\lambda \le C\}$ is compact for any $C$, and hence contains finitely many lattice points.
\end{proof}

\begin{lem}
    For each $u \in \NE_A$, the number
    $$k(u) := \sup\{k: \exists \,u_0<u_1<\ldots <u_k=u\}$$
    is finite. (Here `$u < v$' means `$u \le v$ and $u \neq v$'.)
\end{lem}
\begin{proof}
    If $u_0<u_1<\ldots <u_k=u$, then all $u_i$ are distinct and $\le u$, because $\le$ is a partial order by Lemma \ref{lem:porder}. 
    Hence $k(u)$ is bounded above by $\#\{v:v \le u\}$, which is finite by Lemma \ref{lem:porder}. 
\end{proof}
    
\begin{proof}[Proof of Proposition \ref{prop:vers}]
We assume inductively that $\delta(F)_u = 0$ for all $u$ such that $k(u) < k$. 
The base case $k=1$ holds by assumption on $F_0$. 

We achieve the inductive step by modifying $\psi_p$ and $F$ in such a way that $\delta(F)_u = 0$ for $u$ such that $k(u) = k$ (and $\delta(F)_u = 0$ is unaffected for $k(u) < k$). 
For any $u$ with $k(u)=k$, let us take the $\nov^u$-Taylor coefficient of the equation $\eps(\delta(F)) = 0$ from Lemma \ref{lem:epsdel}. 
It gives
$$ \sum_{v+w=u} \eps_v(\delta(F)_w) = 0.$$
Now $v+w = u$ implies $w \le u$, so $k(u) \ge k(w)$ with equality if and only if $u=w$. 
As $\delta(F)_w = 0$ for all $k(w)<k = k(u)$ by the inductive assumption, the only nonvanishing term is $v=0$, $w=u$, which gives $\eps_0(\delta(F)_u) = 0$. 
Thus we have a cohomology class 
$$[\delta(F)_u] \in H^{2+y(u)}(C_0).$$

\paragraph{Step $1_k$: Arrange that $\delta(F)_u = 0$ for all $u$ such that $y(u) = y_p$ for some $p$.}  
Because $\HH^{2+y_p}(\cB_0)$ is spanned by the first-order deformation class $b_p$ of $\cB$ by assumption, and $\cL^1_{F_0}$ is an isomorphism, we have
$$[\delta(F)_u] = \cL^1_{F_0} (c_{u} \cdot b_p)$$
for some $c_u \in \Bbbk$. 
Because $\NE_A$ satisfies \textbf{(Nice $3_p$)}, $u - e_p \in \NE_A$. Thus we may modify $\psi_{p,u-e_p} \mapsto \psi_{p,u-e_p} + c_u$. 
This has the effect of modifying
$$\mu_{\Psi^* \cB} \mapsto \mu_{\Psi^*\cB} + c_u\nov^u \cdot \mu_{\cB,e_p} + o(\nov^u),$$
where `$o(\nov^u)$' means a sum of terms $\nov^{u+v}$ with $v \in \NE_A \setminus \{0\}$. Note that $k(u+v) > k(u) = k$ when $v \in \NE_A\setminus \{0\}$.
As a result, it has the effect of modifying $\delta(F) = \mu_{\cA}\{\}_F - F\{\mu_{\Psi^*\cB}\}$ by
\begin{align*}
\delta(F) & \mapsto \delta(F) - c_u\nov^u \cdot F_0\{b_p\} + o(\nov^u)\\
& = \delta(F) - c_u\nov^u \cL^1_{F_0} \left(b_p\right) + o(\nov^u).
\end{align*}
In particular, we have arranged $[\delta(F)_u] = 0$, without altering $\delta(F)_v$ for any $v \neq u$ with $k(v) \le k$.

We now choose $f_u \in C_0^{1+y_p}$ such that $\partial f_u = \delta(F)_u$. 
We now modify $F \mapsto F-f_u \nov^u$.  
One may easily check directly that this has the effect of modifying
$$\delta(F) \mapsto \delta(F) - \partial f_u\nov^u +o(\nov^u).$$
In particular, we have $\delta(F)_u = 0$ after this modification, without altering $\delta(F)_v$ for any $v \neq u$ with $k(v) \le k$.

\paragraph{Step $2_k$: Arrange that $\delta(F)_u = 0$ for all $u$ such that $y(u) \notin \{0\} \cup \{y_p\}_{p \in P}$.} 
Observe that $[\delta(F)_u] \in \HH^{2+y(u)}(C_0) \cong \HH^{2+y(u)}(\cA_0) = 0$ by assumption (as $u \in \NE_A$, so $y(u) \in y(\NE_A)$). Thus we may modify $F$ so that $\delta(F)_u = 0$, as in Step $1_k$.

\paragraph{Step $3_k$: Arrange that $\delta(F)_u = 0$ for all $u$ such that $y(u) = 0$.} 
Taking the $\nov^{u+e_p}$ Taylor coefficient of the equation $\eps(\delta(F)) = 0$ from Lemma \ref{lem:epsdel}, we have
$$\sum_{v,w \in \NE_A, v+w=u+e_p} \eps_v(\delta(F)_w) = 0.$$
We claim that the only non-zero terms are $(v,w) = (0,u+e_p)$ and $(e_p,u)$. 
Indeed, when $v \neq 0$, we have $w \le u$ by definition; and if furthermore $v \neq e_p$ then $w < u$, so $k(w)<k(u) = k$, which implies that $\delta(F)_w = 0$ by the inductive hypothesis. 
Thus we obtain
$$\eps_{e_p}(\delta(F)_u) + \eps_0(\delta(F)_{u+e_p}) = 0,$$
showing that $\eps_{e_p}(\delta(F)_u)$ is exact. 
As $\delta(F)_{e_p}=0$ by Step $1_2$ (note that this always precedes Step $3_k$ for any $k \ge 2$), Lemma \ref{lem:eps_obs} implies that $[\delta(F)_u]$ lies in $\cR^1_{F_0}(\ker(Obs)) = 0$. 
Thus $[\delta(F)_u] = 0$, so we may modify $F$ to arrange that $\delta(F)_u=0$ as in Steps $1_k$, $2_k$.

This completes the inductive construction; note that our successive modifications to $\Psi$ and $F$ converge, by $\fm$-adic completeness.
There remains one final thing to check, at the first non-trivial step of the induction $k=2$. 
Namely, we need to ensure that $\psi_p(0) \in \Bbbk$ is a unit, so that $\psi_p$ are indeed units in $R_0$. 
For this we observe that we choose $\psi_p(0)$ so that
$$\cR^1_{F_0}(a_p) = \psi_p(0) \cdot \cL^1_{F_0}(b_p).$$
As $\cR^1_{F_0}$ and $\cL^1_{F_0}$ are isomorphisms, and both $b_p$ and $a_p$ generate the corresponding graded piece which is free of rank $1$ by assumption, we conclude that $\psi_p(0)$ is a unit as required. 
This completes the proof of Proposition \ref{prop:vers}.
\end{proof}

\section{Tropical regularity criterion}\label{sec:tropreg}

The purpose of this section is to prove a standard tropical criterion for smoothness of a hypersurface in a toric variety, essentially by combining \cite[Proposition 4.5.1]{Maclagan2007} and \cite[Corollary 3.1.7]{Batyrev1993}. 
Let $\BbK$ be a field, and $\val: \BbK \to \R \cup \{\infty\}$ a non-Archimedean valuation, which is equal to $0$ on the image of the natural homomorphism $\Z/\charac{\BbK}  \to \BbK$, except for $0$ which of course has valuation $\infty$.

Let $\Sigma^*$ be a complete fan in $M$, and $Y^*_\BbK$ the corresponding toric variety over $\BbK$. 
Let $\cL_{\Delta^*}$ be the ample line bundle over $Y^*_\BbK$ corresponding to the convex polytope $\Delta^*$ in $M^*_\R$, with a basis of sections $z^{\vec{p}}$ indexed by the set $\Xi$ of lattice points in $\Delta^*$. 
Note that in the proof of Lemma \ref{lem:sing_iso}, we will have $Y^*_\BbK = \mathbb{A}^I_{\BbK}$. 

We consider a hypersurface 
$$X^*_b = \left\{\sum_{\vec{p} \in \Xi} b_{\vec{p}} \cdot z^{\vec{p}} = 0\right\} \subset Y^*_\BbK.$$
We consider the function
\begin{align*}
v: \Xi & \to \R\cup\{\infty\},\\
v(\vec{p}) &= \val(b_{\vec{p}}).
\end{align*}
Let $\psi: \Delta^* \to \R \cup\{-\infty\}$ be the smallest convex function such that $\psi(\vec{p}) \ge -v(\vec{p})$. 
We will assume the function $\psi$ is finite, which is equivalent to $v(\vec{p}) \neq \infty$ for all vertices $\vec{p}$ of $\Delta^*$. 
The decomposition into domains of linearity of $\psi$ induces a subdivision of $\Delta^*$ into polytopes. 

\begin{prop}\label{prop:tropreg}
Suppose that either $\BbK$ is algebraically closed, or $\BbK = \Lambda_{\Bbbk,Q}$; $\Sigma^*$ is smooth; and the subdivision of $\Delta^*$ induced by $v$ is a decomposition into simplices, all of which have normalized affine volume which is not divisible by $\charac(\BbK)$, and which only intersect $\Xi$ at their vertices. Then $X^*_b$ is smooth.
\end{prop}
\begin{proof}
It suffices to prove the case when $\BbK$ is algebraically closed; the case $\BbK = \Lambda_{\Bbbk,Q}$ then follows from the case $\BbK = \bar{\Lambda}_{\Bbbk,Q} = \Lambda_{\bar{\Bbbk},\bar{Q}}$ where $\bar{\Bbbk}$ is the algebraic closure of $\Bbbk$ and $\bar{Q}$ is the saturation of $Q$ in $\R$. 

Suppose then, to the contrary, that $\BbK$ is algebraically closed and $X^*_b$ is not smooth. 
Then there exists a non-smooth closed point $x \in X^*_b$. 
We assume that it lies in the toric orbit corresponding to the cone $\sigma$ of $\Sigma^*$. 
Then $x$ lies in the Zariski-open chart $\spec\left( \BbK[\sigma^\vee]\right)$ of $Y^*_\BbK$. 
We choose a basis $\{e_i\}_{i=1,\ldots,n}$ for $M^*$ which is contained in $\sigma^\vee$, and such that $\{e_i\}_{i=1,\ldots,k}$ is a basis for the largest linear subspace contained in $\sigma^\vee$. 
This induces an isomorphism $\BbK[\sigma^{\vee}] \simeq \BbK[z_1^{\pm 1},\ldots,z_k^{\pm 1},z_{k+1},\ldots,z_n]$.
The intersection of $X^*_b$ with this chart is cut out by the equation $f=0$, 
$$ f(z) = \sum_{\vec{p} \in \Xi} b_{\vec{p}} \cdot z^{\vec{p} - \vec{p}_0},$$
where $\vec{p}_0$ is a lattice point lying on the linear subspace $F_\sigma$ supporting the face of $\Delta^*$ dual to $\sigma$.  

As $x=(x_1,\ldots,x_k,0,\ldots,0)$ is a singular point of $X^*_b$, we have $f(x)=0$ and
\begin{align}
\nonumber z_i \frac{\partial f}{\partial z_i}(x) &=0 \qquad \text{for $i=1,\ldots,k$} \\
\nonumber\Rightarrow \sum_{\vec{p} \in \Xi} b_{\vec{p}} \cdot (\vec{p}-\vec{p}_0)_i \cdot x^{\vec{p} - \vec{p}_0} &=0 \qquad \text{for $i=1,\ldots,k$}\\
\label{eq:derivs}\Rightarrow \sum_{\vec{p} \in \Xi} b_{\vec{p}} \cdot x^{\vec{p} - \vec{p}_0} \cdot (\vec{p}-\vec{p}_0) &=0.
\end{align}
We now let $\val(x) = (\val(x_1),\ldots,\val(x_k)) \in \R^k$, and consider the affine linear functions
\begin{align*}
u_{\vec{p}}: \R^k & \to \R\\
u_{\vec{p}}(w) & = v(\vec{p}) + \langle w,\vec{p}-\vec{p}_0\rangle
\end{align*} 
for $\vec{p}$ lying on $F_\sigma$. 
We observe that
$$ \val \left(b_{\vec{p}}\cdot x^{\vec{p}-\vec{p}_0}\right) = 
\left\{ \begin{array}{ll}
    u_{\vec{p}}(\val(x)) & \text{if $\vec{p}$ lies on $F_\sigma$}\\
    \infty & \text{otherwise.}
    \end{array} \right.
$$

Consider the set
\begin{align*}
A_w &= \left\{\vec{p} \in\Xi: u_{\vec{p}}(w)\text{ minimal}\right\} \qquad \text{for $w \in \R^k$}.
\end{align*}
Note that by our assumption that $v(\vec{p})$ is finite at the vertices of $\Delta^*$, the function $f$ does not vanish along the toric orbit containing $x$; therefore, by the non-Archimedean triangle inequality, $f(x) = 0$ implies $\# A_{\val(x)} \ge 2$. 

By our assumption on the subdivision induced by $v(\vec{p})$, and Legendre duality, $A_{\val(x)}$ is the set of vertices of a simplex. 
We may choose $\vec{p}_0$ to be one of the vertices, so that the simplex is contained in a linear (not just affine linear) subspace. 
Let $\vec{p}'$ be a non-zero vertex of $A_{\val(x)}$ (which exists as $\#A_{\val(x)} \ge 2$). 
Then there exists $\vec{q} \in M$ such that $\langle \vec{q},\vec{p}-\vec{p}_0\rangle$ vanishes for all vertices $\vec{p}$ of the simplex except for $\vec{p}'$, and $\langle\vec{q},\vec{p}'-\vec{p}_0\rangle$ is the affine distance from $\vec{p}'-\vec{p}_0$ to the opposite face. 
This distance multiplied by the normalized affine volume of the opposite face gives the normalized affine volume of the simplex, which is not divisible by $\charac(\BbK)$ by hypothesis; in particular, $\langle \vec{q},\vec{p}'-\vec{p}_0\rangle$ is non-zero in $\BbK$, and therefore has zero valuation by our assumption on $\BbK$. 

It follows that when we pair $\vec{q}$ with equation \eqref{eq:derivs}, the term $\vec{p} = \vec{p}'$ is the unique one with the minimal valuation, and that valuation is not $+\infty$; however the sum of terms should vanish, which contradicts the non-Archimedean triangle inequality. 
Therefore $X^*_b$ is smooth.
\end{proof}

\section{Computation of the mirror map}
\label{sec:comp_mm}

In this appendix, we outline how to explicitly compute the mirror map using \cite[Section 6.3.4]{coxkatz} by reducing the calculation to a series of lemmas that follow from direct computations. Compare \cite{Adolphson2014}.
 
Let $\mathcal{A} = (P \cup \{0\}) \times \{1\} \subset M \oplus \Z$, and $\hat\beta = (0,-1) \in M \oplus \Z$. 
The lattice of relations of $\mathcal A$ is isomorphic to $K$, via
\begin{align*}
    K& \to \Z^{P \cup \{0\}} \\
    k&\mapsto (k,-|k|)
\end{align*}
where $|k|:=\sum_i k_i$ for $k\in \Z^P$. 

The $\mathcal A$-GKZ system, as defined in \cite[Section 5.5]{coxkatz}, is a system of differential equations for a function $\phi(\snov)$ of variables $\snov = (\snov_{\vec{p}})_{\vec{p} \in \mathcal{A}}$:
\begin{align}
\label{eq:GKZ_Z}    Z_v\phi(\snov) &= v(\hat\beta) \qquad \text{ for $v \in (M \oplus \Z)^\vee$,}\\
\label{eq:GKZ_square}    \square_k \phi(\snov) &= 0 \qquad \text{ for $k \in K$,}
\end{align}
where 
\begin{align*}
    Z_v &= \sum_{\vec{p} \in \mathcal A} v(\vec{p}) \cdot \snov_{\vec{p}} \frac{\partial}{\partial \snov_{\vec{p}}},\\
    \square_k &= \partial^{k_+}- \partial ^{k_-}, \qquad \text{where}\\
    \partial^{k_\pm}&= \prod_{\vec{p}:\pm k_{\vec{p}} >0} \partial_{\snov_{\vec{p}}}^{\pm k_{\vec{p}}}. 
\end{align*}

By \cite[Section 6.3.4]{coxkatz} (see also \cite{Batyrev1995}), one computes the mirror map as follows. 
We set $\nov_{\vec{p}} = -\snov_{\vec{p}}/\snov_{0}$, for $\vec{p} \in P$. 
Find a holomorphic function $\tau(\nov)$ such that $\snov_0^{-1}\tau(\nov)$ satisfies equations \eqref{eq:GKZ_Z}, \eqref{eq:GKZ_square}.
Next, for $u \in K$, find a holomorphic function $\tau_u(\nov)$ such that $\snov_0^{-1}(\tau(\nov) \log(\nov^u) + \tau_u(\nov))$ satisfies the same equations, and $\tau_u$ has vanishing constant term. 
Then the mirror map sends 
$$\nov^u \mapsto \nov^u \cdot \exp\left(\frac{\tau_u(\nov)}{\tau(\nov)}\right).$$

\begin{lem}
    Let
   \[ \tau(\nov)= \sum_{u\in K_{\geq 0}} \mathsf{comb}(u) \cdot \nov^u.\]
   Then $\snov_0^{-1} \tau(\nov)$ is a solution to \eqref{eq:GKZ_Z} and \eqref{eq:GKZ_square}.
\end{lem}

\begin{lem}\label{lem:taup}
    Let  
    \begin{align*}
        \tilde\tau_{\vec{p}}(\nov) & = 
        \sum_{u\in K_{\geq 0}} \mathsf{comb}(u) \cdot  (-H_{u_{\vec{p}}})\cdot \nov^u, \\
        \gamma_{\vec{p}}(\nov)=&\sum_{u\in K_{\vec{p}} \setminus K_{\ge 0}} (-1)^{u_{\vec{p}}+1}\cdot \mathsf{comb}_{\vec{p}}(u) \cdot \nov^u.
    \end{align*}
    Then $\snov_0^{-1} (\tau(\nov)\cdot\log(\snov_{\vec{p}}) + \tilde\tau_{\vec{p}}(\nov)+\gamma_{\vec{p}}(\nov))$ is a solution to \eqref{eq:GKZ_square}, and $\tilde\tau_{\vec{p}}+\gamma_{\vec{p}}$ has vanishing constant term.
\end{lem}

\begin{lem}\label{lem:tau0}
    Let
    \begin{align*}
        \tilde\tau_0(\nov) &= \sum_{u\in K_{\geq 0}} \mathsf{comb}(u) \cdot  (-H_{|u|}) \cdot \nov^u.
    \end{align*}
    Then $\snov_0^{-1} (\tau(\nov)\cdot\log(\snov_0) + \tilde\tau_{0}(\nov))$ is a solution to \eqref{eq:GKZ_square}, and $\tilde\tau_0$ has vanishing constant term.
\end{lem}

We observe that $\tau_{\vec{p}} = \tilde \tau_{\vec{p}} - \tilde\tau_0$.

The proof of Lemmas \ref{lem:taup} and \ref{lem:tau0} makes use of the following elementary Lemma:

\begin{lem}
    We have
    \[\partial_x^u(x^a \log (x)) - \partial^u_x(x^a)\log(x)= \begin{cases}
        \frac{a!}{(a-u)!}\cdot (H_a-H_{a-u})\cdot x^{a-u} & \text{ if $a\geq u$}\\
        (-1)^{u-a+1} \cdot a! \cdot (u-a-1)! \cdot x^{a-u} & \text{ if $0 \le a<u$}\\
        (-1)^u\cdot\frac{(u-a-1)!}{(-a-1)!} \cdot (H_{u-a-1} - H_{-a-1}) \cdot x^{a-u}& \text{ if $a<0$}
    \end{cases}\]
\end{lem}

\begin{cor}
    For any $u \in K$, let 
    \begin{align*}
        \tau_u(\nov) &= -|u| \cdot \tilde\tau_0(\nov) + \sum_{\vec{p} \in P} u_{\vec{p}} \cdot \tilde\tau_{\vec{p}}.
    \end{align*}
    Then $\snov_0^{-1}(\tau(\nov)\cdot \log(\nov^u) + \tau_u(\nov))$ is a solution to \eqref{eq:GKZ_square} and \eqref{eq:GKZ_Z}, and $\tau_u$ has vanishing constant term.
\end{cor}

\begin{cor}
    The mirror map is given by $\Phi|_{\C[[K_{\ge 0}]]}$, where $\Phi$ is as in \eqref{eq:mm_formula}.
\end{cor}

\renewbibmacro{in:}{}
\def\bibrangedash{ -- }
\printbibliography

\end{document}

%% file: preamble.tex
\usepackage{tikz-cd}
\usepackage{enumerate}
\usepackage{todonotes}
\usepackage{mathtools}
\usepackage{verbatim}
\usepackage{amssymb}
\usepackage[citestyle=alphabetic,bibstyle=alphabetic,backend=bibtex, maxbibnames=10,minbibnames=5]{biblatex}
\numberwithin{equation}{section}

\newtheorem{conjmain}[mainthm]{Conjecture}

%% file: mathpreamble.tex
\def\BbK{\mathbb{K}}
\def\G{\mathbb{G}}
\def\CP{\mathbb{CP}}

\def\cL{\mathcal{L}}

\def\hookar{\ar@{^{(}->}}
\def\spec{\mathsf{Spec}}

\def\val{\mathrm{val}}

\def\grmf{\mathsf{GrMF}}

\def\fuk{\EuF}
\def\bc{\mathsf{bc}}

\def\Nef{\mathsf{N}}

\def\eps{\epsilon}

\def\mir0{F}

\newcommand{\fm}{\mathfrak{m}}

\newcommand{\ol}[1]{\overline{#1}}

\def\G{\mathbb{G}}

\def\and{\, \& \,}

\def\nov{r}
\def\snov{s}

\def\lcm{\mathsf{lcm}}
\renewcommand{\vec}[1]{\mathsf{#1}}

\def\cL{\mathcal{L}}

\def\cA{\mathsf{A}}
\def\cB{\mathsf{B}}
\def\cC{\mathcal{C}}
\def\cD{\mathcal{D}}
\def\cR{\mathcal{R}}

\def\NE{\mathrm{NE}}
\def\nufun{nu\text{-}fun}
\def\NuFun{Nu\text{-}Fun}
\def\mfun{pre\text{-}fun}
\def\ob{\mathrm{Ob}\,}
\def\Hom{\mathrm{Hom}}
\def\HH{\mathrm{HH}}
\def\charac{\mathrm{char}\,}